\documentclass[11pt]{article}
\usepackage{xcolor,color,enumerate}
\usepackage[color,all]{xy}
\usepackage{multirow}
\usepackage{amsthm}
\usepackage[utf8]{inputenc}
\usepackage{setspace}
\usepackage{amsmath,amsfonts,amssymb}
\usepackage{graphicx,xcolor}

\usepackage{hyperref}

\newcommand{\tr}{\mbox{ tr}}

\newcommand{\ba}{\begin{array}}
\newcommand{\ea}{\end{array}} 

\newtheorem{proposition}{Proposition}[section]
\newtheorem{theorem}[proposition]{Theorem}
\newtheorem{lemma}[proposition]{Lemma}
\newtheorem{corollary}[proposition]{Corollary}

\theoremstyle{definition}
\newtheorem{definition}[proposition]{Definition}
\newtheorem{example}[proposition]{Example}

\theoremstyle{remark}

\numberwithin{equation}{section}
\numberwithin{proposition}{section}

\begin{document}

\centerline{\Large \bf A cohomological approach to immersed }
\vskip 0.25cm
\centerline{\Large \bf  submanifolds via integrable systems}
\vskip 0.5cm
\centerline{A.M. Grundland$^{\dagger}$ and  J. de Lucas$^{\ddagger}$}
\vskip 0.6cm
\centerline{$^{\dagger}$Centre de Recherches Math\'ematiques, Universit\'e de Montr\'eal, }
\centerline{Montr\'eal CP 6128 (QC) H3C 3J7, Canada}
\centerline{$^{\dagger}$Department of Mathematics and Computer Science, Universit\'e du Qu\'ebec,}
\centerline{Trois-Rivi\`eres CP 500 (QC) G9A 5H7, Canada}
\vskip 0.4cm
\centerline{$^{\ddagger}$Department of Mathematical Methods in Physics, University of Warsaw,}
\centerline{ul. Pasteura 5, Warszawa 02-093, Poland}

\begin{abstract}
A geometric approach to immersion formulas for soliton surfaces is provided through new cohomologies on spaces of special types of $\mathfrak{g}$-valued differential forms. This leads us to introduce  Poincar\'e-type lemmas for these cohomologies, which appropriately describe the integrability conditions of Lax pairs associated with systems of PDEs. Our methods clarify the structure and properties of the deformations and soliton surfaces for the aforesaid Lax pairs. Our findings also allow for the generalization of the theory of soliton surfaces in Lie algebras to general soliton submanifolds. Techniques from the theory of  infinite-dimensional jet manifolds and diffieties enable us to justify certain common assumptions of the theory of soliton surfaces. Theoretical results are illustrated through $\mathbb{C}P^{N-1}$ sigma models.
\end{abstract}

\noindent{\it Keywords:} cohomology, $\mathbb{C}P^{N-1}$ sigma model,  generalized symmetries, $\mathfrak{g}$-valued differential forms, $\mathfrak{g}$-valued de Rham cohomology, integrable systems, immersion formulas, soliton surfaces

\medskip

\noindent{\it AMS numbers:} {35Q53, 35Q58, 53A05}


\section{Introduction}\label{Introduction}\setcounter{equation}{0}
Integrable models and their continuous deformations under various types of dynamics have produced considerable interest in various branches of mathematics, physics, and biology (see \cite{BE00,Ch83,DGW96,Da91,GPW92,La03,NPW92,OLX99,PS91,Sa94,So52} for details). In fact, the motivation for this research topic came largely from applications like the growth of crystals \cite{NPW92}, quantum field theory models \cite{DGW96,NPW92,OLX99}, or the motion of boundaries between regions of different viscosities and densities \cite{CJZ89,So52}.

The progress in the analytic description of surfaces obtained from nonlinear PDEs, e.g. soliton or constant curvature surfaces \cite{Bo94,Ta95}, has been rapid and has resulted in many new techniques and theoretical approaches. Some of the foremost relevant developments have occurred in the study of soliton surfaces immersed in Lie algebras by using techniques from the theory of completely integrable systems \cite{Ca53,Ci97,Ci07,DV00,DS92,Ei78,FG96,FGFL00,RS00,Sy82,Sy95}. 

The possibility of using a linear spectral problem (LSP) to represent a moving frame on a soliton surface has yielded many findings concerning their intrinsic geometric properties  \cite{BBT06,Bo94,He01,Ko96,Mi86,MSS91}. The spectral parameter in the LSP describes deformations of soliton surfaces preserving their properties in such a way that integrable surfaces come in a family \cite{Bo94}. These surfaces are characterized by fundamental forms whose coefficients satisfy the Gauss-Weingarten and Gauss-Mainardi-Codazzi equations. It has recently proved fruitful to apply such a characterization of soliton surfaces to $\mathbb{C}P^{N-1}$ sigma models via their immersion formulas in Lie algebras. They have also been shown to play an essential role in many other problems of a physical nature (see \cite{BBT06,GG10,GP78,GP12,Ma02,ZM79} and references therein). 

The construction of the soliton surfaces related to the completely integrable $\mathbb{C}P^{N-1}$ sigma model has been accomplished by representing the Euler-Lagrange equations for this model as a conservation law which in turn provides a closed $\mathfrak{su}(N)$-valued differential one-form on the surfaces. This is the so-called {\it generalized Weierstrass formula for immersion} \cite{GY09,KL99}.

The construction of smooth orientable soliton surfaces related to completely integrable models in the sense of  admitting a LSP problem was pioneered by A. Sym \cite{Sy82,Sy95}. His technique exploits the conformal invariance of the zero-curvature representation of the LSP relative to the spectral parameter \cite{Bo94}. Another approach for determining such surfaces, formulated by Cie\'sli\'nski and Doliwa \cite{Ci97,Ci07,DS92}, is based on the use of gauge symmetries of the LSP. Fokas and Gel'fand \cite{FG96,FGFL00} developed a third approach by  using  the LSP for integrable systems and their Lie symmetries to derive families of soliton surfaces. In all these cases soliton surfaces are described through the so-called {\it immersion formulas}. Most recently, a reformulation and extension of the Fokas-Gel'fand immersion formula has been performed through the formalism of generalized vector fields and their actions on jet spaces in \cite{Grundland15,GLM16,GP11,GP12}. This extension has provided the necessary and sufficient conditions for the existence  of soliton surfaces in terms of the symmetries of the LSP and integrable models. It also described the relations between the previously known immersion formulas.

This paper aims to provide a unifying approach to immersion formulas via cohomological and geometric techniques. Our procedure provides simple geometric proofs for the theoretical results in the previous literature on the topic, simplifies expressions for immersion formulas via geometric structures, and extends the formalism of immersion formulas to create a theory of immersion formulas for soliton submanifolds and generalized Lax pairs with many potential applications (see e.g. \cite{Ablo82,Br13,Mana06,Xiu16,Zakh94} and references therein). 
 
Let $\mathfrak{g}$ be a Lie algebra and let $M,N$ be manifolds. Consider an integrable system of PDEs whose independent and dependent variables are coordinates on $M$ and $N$ respectively. The first idea of the paper is to give a more precise description of the standard objects appearing in the study of soliton surfaces for our system of PDEs, e.g. its LSPs, deformations, and immersion formulas, by the hereafter defined {\it parametrized $\mathfrak{g}$-valued differential forms on-shell} on $M$. In a nutshell, a parametrized $\mathfrak{g}$-valued differential form on-shell is a family of $\mathfrak{g}$-valued differential forms on $M$ parametrized by particular solutions of the integrable system of PDEs and the spectral parameter $\lambda$ of a corresponding LSP. More generally, the {\it parametrized $\mathfrak{g}$-valued differential forms} on $M$ are $\mathfrak{g}$-valued differential forms on $M$ described by arbitrary functions from $M$ to $N$ that need not be particular solutions of our initial system of PDEs. The parametrized $\mathfrak{g}$-valued differential forms will also play a relevant role in the description of Lax pairs and immersion formulas. If not otherwise stated, it is hereafter assumed for simplicity that the independent variables of all systems of PDEs and types of $\mathfrak{g}$-differential forms are defined on $M$.

The space of parametrized $\mathfrak{g}$-valued differential forms, say $\Omega_\mathcal{N}(M)\otimes \mathfrak{g}$, is endowed with a differential ${\bf d}$ leading to a cochain complex (see \cite{Ma99} for details on cochain complexes)
\begin{equation}\label{chain1}
C_\mathcal{N}^\infty(M)\otimes \mathfrak{g}\stackrel{{\rm\bf d}}{\longrightarrow} \Omega_\mathcal{N}^1(M)\otimes \mathfrak{g}\stackrel{{\rm\bf d}}{\longrightarrow}\Omega_\mathcal{N}^2(M)\otimes\mathfrak{g}\stackrel{{\rm\bf d}}{\longrightarrow}\ldots 
\end{equation}
where $\Omega^k_\mathcal{N}(M)\otimes \mathfrak{g}$, with $k\in \mathbb{Z}$, is the space of parametrized $\mathfrak{g}$-valued differential $k$-forms.
A LSP problem is then proved to amount to a parametrized $\mathfrak{g}$-valued differential one-form $\omega\in \Omega^1_\mathcal{N}(M)\otimes\mathfrak{g}$. This allows us to define a second operator ${\bf d}_{2\omega}$ on the above cochain complex satisfying the condition that ${\bf d}_{2\omega} \omega$ vanishes exactly on the solutions of the system of PDEs under consideration. We then say that ${\bf d}_{2\omega}\omega=0$ on-shell.
 
The space $\Omega_\mathcal{S}(M)\otimes\mathfrak{g}$ of  parametrized $\mathfrak{g}$-valued differential  forms  on-shell can be endowed with two differentials ${\bf d}$ and ${\bf d}_{2\omega}$ inducing new cochain complexes
\begin{equation}\label{chain2}
C_\mathcal{S}^\infty(M)\otimes \mathfrak{g}\stackrel{{\bf d},{\bf d}_{2\omega}}{\longrightarrow} \Omega_\mathcal{S}^1(M)\otimes \mathfrak{g}\stackrel{{\bf d},{\bf d}_{2\omega}}{\longrightarrow}\Omega_\mathcal{S}^2(M)\otimes\mathfrak{g}\stackrel{{\bf d},{\bf d}_{2\omega}}{\longrightarrow}\ldots 
\end{equation}
where $\Omega^k_\mathcal{S}(M)\otimes \mathfrak{g}$, with $k\in \mathbb{Z}$, is the space of parametrized $\mathfrak{g}$-valued differential $k$-forms on-shell.
In other words, ${\bf d}^2_{2\omega}=0$ and ${\bf d}^2=0$ on $\Omega_{\mathcal{S}}(M)\otimes\mathfrak{g}$. We extend the Poincar\'e Lemma for the standard de Rham cohomology to this new realm so as to prove that every element $\vartheta\in \Omega^k_\mathcal{S}(M)\otimes \mathfrak{g}$ is {\it ${\bf d}_{2\omega}$-closed}, i.e. ${\bf d}_{2\omega}\vartheta=0$, if and only if it is locally exact, i.e. there exists for any arbitrary point $p\in M$ an open $U\subset M$ containing $p$ and $\varphi\in \Omega^{k-1}_\mathcal{S}(U)\otimes \mathfrak{g}$ such that $\vartheta={\bf d}_{2\omega}\varphi$ on $U$. 

Subsequently, we prove that a deformation of the LSP described by $\omega$ is equivalent to a ${\bf d}_{2\omega}$-closed parametrized $\mathfrak{g}$-valued differential one-form on-shell $\Upsilon$ , i.e. ${\bf d}_{2\omega}\Upsilon=0$. The immersion formulas of Cie\'sli\'nski--Doliwa and Fokas--Gel'fand provide particular solutions $\Upsilon$ of this equation for  $M=
\mathbb{R}^2$. Since ${\bf d}_{2\omega}^2=0$, every ${\bf d}_{2\omega}S$, with $S
\in C_\mathcal{S}^\infty(M)\otimes \mathfrak{g}$, is a ${\bf d}_{2\omega}$-closed parametrized $\mathfrak{g}$-valued differential one-form on-shell. Hence, every $\Upsilon\in \Omega^1_\mathcal{S}(M)\otimes \mathfrak{g}$ such that $\Upsilon:={\bf d}_{2\omega}S$ gives rise to a deformation of the LSP related to $\omega$. A very simple lemma allows us to retrieve  the Cie\'sli\'nski-Doliwa immersion formula from this latter fact. The Fokas-Gel'fand and Sym-Tafel immersion formulas can be explained in a slightly more involved but similar way. In general, their results can easily be retrieved from simple geometric considerations.

Geometrically, we show that the immersion formulas induced by an infinitesimal deformation $\Upsilon$ take the form $F={\rm Ad}_{\Phi^{-1}}\Upsilon$, where $\Phi$ stands for a family of solutions for the LSP induced by $\omega$ parametrized by its spectral parameter and solutions of the associated system of PDEs. Since the corresponding deformations are given by a ${\bf d}_{2\omega}$-closed $\Upsilon$, our extended Poincar\'e Lemmas ensure that $\Upsilon$ is locally exact, i.e. $\Upsilon={\bf d}_{2\omega}S$ for an $S\in C_\mathcal{S}^\infty(M)\otimes \mathfrak{g}$. The function $S$ permits us to obtain the immersion formula associated with $\Upsilon$ for all known immersion formulas. For instance, our simple Lemma \ref{lem:closed} enables us to retrieve the Cie\'sli\'nski-Doliwa immersion formula immediately as $F={\rm Ad}_{\Phi^{-1}}{\bf d}_{2\omega}S={\bf d}{\rm Ad}_{\Phi^{-1}}S$.

Different types of systems of PDEs appearing in the study of immersion formulas can be described as different types of parametrized $\mathfrak{g}$-valued differential forms. Our cohomologies describe the integrability conditions for such systems of PDEs. Local solutions are related to closed forms and global solutions to exact ones. Hence, the cohomology groups related to our cohomologies reflect the nature of the solutions of PDEs. 

Results on the characterization of $\lambda$-conformal symmetries, symmetries of the system of PDEs, the existence of gauges for immersion formulas, and the presence of gauge transformations are easily proved by using our cohomologies. 

The parametrized $\mathfrak{g}$-valued differential forms can be generated by introducing a class of $\mathfrak{g}$-valued differential forms on jet spaces. An analogue of our cohomological procedures is developed to deal with this jet approach. This requires the use of techniques from the theory of infinite-dimensional jet manifolds and diffities \cite{Vi01}. As a particular case, this retrieves the jet approach to immersion formulas given in \cite{GLM16}. 

Finally, all previous results are coordinate-free and they do not depend on the base manifold $M$. Hence, the whole theory can be extended to immersion formulas and Lax pairs on any $M$, e.g. to immersion formulas for submanifolds in Lie algebras. 

This paper is organized as follows. In Section 2  we briefly summarize some relevant results on jet bundles and the geometry of differential equations to be used hereafter. Section 3 is concerned with the description of the new cohomologies proposed in this work. Sections 3 and 4 propose and analyse useful new types of $\mathfrak{g}$-valued differential forms in the theory of immersion formulas and soliton submanifolds. The basic facts on LSPs and their description in terms of our new types of $\mathfrak{g}$-valued differential forms are introduced in Section 5. Subsequently, the theory of immersion formulas for general soliton submanifolds through cohomological and geometrical techniques is addressed in Section 6. The relation between different immersion formulas, e.g. the previously described multidimensional generalizations of the Sym--Tafel (ST), Cie\'sli\'nski--Doliwa (CD) and Fokas--Gel'fand (FG) formulas, are detailed in Section 7. Next, Section 8 concerns the special case of immersion formulas related to Lax pairs describing solutions of PDEs obeying a certain boundary condition. Our techniques are illustrated through soliton surfaces for $\mathbb{C}P^{N-1}$ sigma models  in Section 9. The main contributions of the paper and topics for further research are summarized in Section 10.

\section{On classical and generalized Lie symmetries}\label{ClaLie}\setcounter{equation}{0}
This section reviews the main notions to be used hereafter: finite and infinite-dimensional jet bundles, Lie symmetries, and related properties (see \cite{Ol93,Vi01,KV99} for details).  The summation convention over repeated indices is hereafter employed. 

As previously, $M$ and $N$ are differential manifolds whose coordinates stand for the independent and dependent variables, respectively, of the system of partial differential equations under consideration. We write $\overline{a,b}:=\{a,a+1,\ldots,b-1,b\}$ for $a<b\in \mathbb{Z}$, and we assume $M$ and $N$ to have local coordinates $x^\alpha$, for $\alpha\in\overline{1,m}$ and $u^i$, with $i\in \overline{1,n}$, correspondingly.

Let $J^p$ be the $p$-order jet bundle related to the trivial bundle $(N\times M, M,\pi)$ where $\pi:
N\times M\rightarrow M$ is the projection onto $M$. As usual, we define $J^0:=M\times N$. The coordinates of $J^p$ are given by $x^\alpha$, $u^i$, and $u^i_J$,  where $J:=(j_1,\ldots,j_p)$ is a  multi-index with $1\leq j_1\leq \ldots\leq j_p\leq m$ and length $|J|=p$. If $|J|=0$, then we set $u^i_J:=u^i$. We denote by $j^p_xu:=(x,u_J)$ a generic element of $J^p$. Given a section $s\in \Gamma(\pi)$, say $s(x):=(x,u(x))$, its {\it prolongation} to $J^p$ is denoted by $j^ps$ and reads $j^ps(x):=(x,u(x),{\rm d}u/{\rm d}x(x),\ldots, {\rm d}^pu(x)/{\rm d}x^p)$.  Every function $u(x):M\rightarrow N$ amounts to a section $s_u:x\in M\mapsto (x,u(x))\in M\times N$ of $J^0$ and vice versa.  This motivates us to identify $u$ and its related section $s$ so as to simplify the notation. The main geometric structure on $J^p$, the so-called {\it Cartan distribution} $\mathcal{C}^p$, is the smallest distribution on $J^p$ tangent to all prolongations $j^ps$ for an arbitrary $s\in\Gamma(\pi)$ \cite{Ol93}.

The space $J^\infty$ is assumed to be the space given by the inductive limit of all $J^p$. The space of smooth functions on $J^\infty$ reads $C^\infty(J^\infty):=\cup_{p\in \overline{\mathbb{N}}}C^\infty(J^p)$ where $\overline{\mathbb{N}}:=\mathbb{N}\cup \{0\}$. In other words, the smooth functions on $J^\infty$ are smooth functions depending on a finite, but undetermined, set of variables $x,u_J$. Previous construction allows us to consider $J^\infty$ as a particular type of manifold. Subsequently, $[u]$ stands for an arbitrary element of $J^\infty$. The definition of the Cartan distribution on finite-dimensional jet bundles can be extended to $J^\infty$ giving rise to an involutive distribution $\mathcal{C}$ on $J^\infty$ (see  \cite{KV11,Vi01} for details). 

A vector field $X$ on $J^0$ can be written in local coordinates as
\begin{equation}
X=\xi^\alpha(x,u)\partial_{\alpha}+\varphi^i(x,u)\partial_{i},
\end{equation}
where $\partial_{\alpha}:={\partial}/{\partial x^\alpha}$, for $\alpha\in \overline{1,m}$, and  $\partial_{i}:={\partial}/{\partial u^i}$ for $i\in \overline{1, n}$. The vector field $X$ gives rise to its so-called {\it prolonged vector field} ${\rm pr}\,X$ on  $J^\infty$: the only vector field on $J^\infty$ leaving the space of vector fields taking values in $\mathcal{C}$ invariant (relative to the Lie bracket of vector fields) and projecting onto $X$. Its expression in coordinates reads \cite{Ol93,KV99}
\begin{equation*}
\hspace{-2cm}{\rm pr}X:=\xi^\alpha\partial_{\alpha}+\varphi_J^i\frac{\partial}{\partial u^i_J},\qquad 
\varphi^i_J:=D_JR^i+\xi^\alpha u^i_{J,\alpha}, \qquad R^i:=\varphi^i-\xi^\alpha u^i_\alpha,
\end{equation*}
where the $R^i$ are the so-called {\it characteristics} of the vector field $X$ and $D_J:=D_{j_1}\ldots D_{j_p}$ for
\begin{equation}
D_\alpha=\partial_{\alpha}+u^i_{J,\alpha}\frac{\partial}{\partial u^i_J},\hspace{8mm}\alpha\in\overline{1,m},
\end{equation}
where $u^i_{J,\alpha}$ represents the variable $u^i_{J'}$ with $J':=(j_1,\ldots, j_p,\alpha)$ and $J$ is an arbitrary multi-index. 
Geometrically, $j^pX$ is the unique vector field on $J^p$ projecting onto $X$ on $M\times N$ that leaves the vector fields taking the Cartan distribution $\mathcal{C}^p$ invariant (with respect to its action by Lie brackets). Alternatively, $j^pX$ can be defined as the restriction of ${\rm pr}\,X$ to functions in $C^\infty(J^p)$.
 
Equivalently, ${\rm pr}\,X$ can be written as
\begin{equation}
{\rm pr}\,X=\xi^\alpha D_\alpha+{\rm pr}\,X_R,\qquad X_R:=R^i\frac{\partial}{\partial u^i}.
\end{equation}

The vector fields on $J^p$ are always related to a uni-parametric group of diffeomorphisms describing their flow. Meanwhile, the vector fields on $J^\infty$ are not generally associated with any such a group of diffeomorphisms \cite{KV99}. Nevertheless, there exist relevant types of vector fields on $J^\infty$ that admit a uni-parametric group of transformations on $J^\infty$, e.g. the evolutionary vector fields. An {\it evolutionary vector field} is a vector field  on $J^\infty$ that projects onto each $J^p$ and whose expression on $J^\infty$ can be considered as the limit of its projections. For instance, the so-called {\it total derivatives} $D_\alpha$ are evolution vector fields \cite{KV99}. 

Let us define on $J^p$ a system of partial differential equations (PDEs) in $m$ independent and $n$ dependent variables of the form 
\begin{equation}
\Delta^\mu(j^p_xu)=0,\qquad \mu=\overline{1,s}\label{PDE}.
\end{equation}
A particular solution of (\ref{PDE}) is a map $u(x)$ from $M$ to $N$ whose prolongation to $J^p$, thought of as a section of $J^0$, satisfies (\ref{PDE}). 

The system of PDEs (\ref{PDE}) determines a region $\mathcal{E}\subset J^p$ where the functions $\Delta^\mu$, with $\mu\in \overline{1,s}$, vanish. The system of PDEs (\ref{PDE}) is hereupon assumed to be {\it locally solvable}, namely for each $j^p_xu\in \mathcal{E}$ there exists a solution $u(x)$ of the system (\ref{PDE}) such that $j^p_xu$ belongs to the prolongation of $u(x)$ to $J^p$ \cite[p. 158]{Ol93}. Additionally, it is also assumed that the system (\ref{PDE}) has {\it maximal rank}, i.e. the functions $\Delta^\mu$ are functionally independent and hence the space $\mathcal{E}$ can be considered as a submanifold of $J^p$ (cf. \cite[p. 158]{Ol93}).  These are assumptions satisfied by a large family of  differential equations.

A vector field $X$ on $J^0$ is a classical Lie point symmetry of the nondegenerate system of PDEs (\ref{PDE}) if the prolongation of $X$ to $J^p$ is such that
\begin{equation}\label{Sys}
j^pX\,\Delta^\mu|_\mathcal{E}=0, \qquad \mu\in \overline{1,s}.
\end{equation}
Therefore, the commutator of two classical Lie point symmetries is a Lie point symmetry. Thus, the Lie point symmetries form a Lie algebra $V_s$, which, if finite-dimensional, locally defines an action of a Lie group $G$ on $J^0$.

The symmetry group $G$ transforms solutions of (\ref{Sys}) into new solutions. This means that the graph corresponding to one solution is transformed into the graph associated with another solution. If the graph is preserved by the group $G$ or, equivalently, if the vector fields $X_a:=\xi_a^\alpha\partial_{\alpha}\alpha+\varphi^i_a\partial_i$, with $a\in \overline{1,r}$, conforming a basis of the Lie algebra $V_s$ are tangent to the graph, then the related solution is said to be {\it $G$-invariant}. Invariant solutions satisfy, in addition to the equations (\ref{PDE}), the characteristic equations equated to zero
\begin{equation}
\varphi^i_a(x,u)-\xi^\alpha_a(x,u)u^i_{,\alpha}=0,\qquad a\in \overline{1,r},
\end{equation}
where the index $a$ runs over the elements of a basis of $V_s$.

The above formalism can be extended to $J^\infty$ by extending the system (\ref{PDE}) on $J^p$ to a system of equations on $J^\infty$ of the form
\begin{equation}
\label{prol}
\Delta^\mu([u]),\qquad D_J\Delta^\mu([u])=0,\qquad  \mu\in \overline{1,s}, \quad\forall J.
\end{equation}
Recall that $\Delta^\mu$ only depends on the variables $u^J$ for $|J|\leq p$. Hence, the solutions of (\ref{prol}) are the extension to  $J^\infty$ of particular solutions of $\Delta^\mu(j^p_xu)=0$. The equations (\ref{prol}) define a submanifold $\mathcal{E}^{\infty}$ of $J^\infty$ and the total derivatives, $D_J$, are now tangent to it.

The {\it Cartan distribution} $\mathcal{C}$ on $J^\infty$ is the distribution generated by the vector fields $D_\alpha$ for $\alpha\in \overline{1,m}$. The Cartan distribution is involutive, i.e. the Lie bracket of vector fields taking values in  $\mathcal{C}$ also takes values in $\mathcal{C}$. The vector fields taking values in the Cartan distribution $\mathcal{C}$ are tangent to  $\mathcal{E}^\infty$, which allows us to restrict $\mathcal{C}$ to $\mathcal{E}^\infty$ giving rise to the distribution $\mathcal{C}|_{\mathcal{E}^\infty}$ on $\mathcal{E}^\infty$. The pair $(\mathcal{E}^\infty,\mathcal{C}|_{\mathcal{E}^\infty})$ is called a {\it diffiety} \cite{Vi01}.

It may happen that invariant solutions of (\ref{PDE}) are restricted in number or trivial if the full symmetry group is small. To extend the number of symmetries, and thus of solutions, one looks for generalized symmetries. They exist only if the nonlinear equation (\ref{PDE}) is {\it integrable} \cite{MSS91}, i.e. it has been obtained as the compatibility of a Lax pair. To describe them, we make use of generalized vector fields. A {\it generalized vector field} is a vector field $X_R$ on $J^\infty$ (\cite[Definition 5.1]{Ol93}) of the form
\begin{equation*}
X_R:=R^i([u])\frac{\partial}{\partial u^i},\label{1.12}
\end{equation*}
where the functions $R^i\in C^\infty(J^\infty)$ are arbitrary. 
The space of generalized vector fields is isomorphic to the space ${\rm sym}(\mathcal{C})$ of symmetries of the Cartan distribution $\mathcal{C}$ modulo {\it Cartan vector fields}, namely vector fields taking values in $\mathcal{C}$ \cite{KV99}. Indeed, every vector field in ${\rm sym}(\mathcal{C})$  gives rise to a generalized vector field by restricting it to $C^\infty(J^0)$ and every generalized vector field $X_R$ can be extended to a unique vector field ${\rm pr}\, X$ on $J^\infty$ leaving invariant the space of vector fields taking values in $\mathcal{C}$, i.e. ${\rm pr}\, X$ induces an element of ${\rm sym}(\mathcal{C})$, whose action on $C^\infty(J^0)$ coincides with $X_R$. In coordinates, 
\begin{equation}
{\rm pr}X_R=X_{R}+\sum_{|J|>0}D_JR^i\frac{\partial}{\partial u_J^i}.
\end{equation}
A vector field $X_R$ is a {\it generalized symmetry} of the system of PDEs (\ref{PDE}) if and only if ${\rm pr}X_R$ is tangent to $\mathcal{E}^\infty$ \cite{Ol93}.

\section{Cohomologies on parametrized $\mathfrak{g}$-valued differential forms}
This section addresses the generalizations of standard cohomologies of differential forms, e.g. the de Rham one, to the realm of cohomologies for differential forms taking values in Lie algebras parametrized by sections of a vector bundle $(M\times N,M,\pi)$ and a parameter. This entails the development of analogues for these new cohomologies of the Poincar\'e Lemma for de Rham cohomology. These results will later play a fundamental role in studying and extending Lax pairs and immersion formulas for soliton surfaces within Lie algebras to immersion formulas for soliton submanifolds in Lie algebras. 

Let ${\rm d}$ stand for the standard exterior differential on differential forms. 
The space $\Omega(M)\otimes \mathfrak{g}$ of 
$\mathfrak{g}$-valued differential forms on $M$ admits a natural exterior differential ${\rm \bf d}:\Omega(M)\otimes \mathfrak{g}\rightarrow \Omega(M)\otimes \mathfrak{g}$ defined by setting ${\rm \bf d}(\theta\otimes v):=({\rm d}\theta)\otimes v$ for every $v\in \mathfrak{g}$ and $\theta\in \Omega(M)$ and extending ${\rm \bf d}$ by $\mathbb{R}$-linearity over the whole $\Omega(M)\otimes \mathfrak{g}$. Due to the definition of ${\bf d}$ in terms of ${\rm d}$, it follows that ${\rm \bf d}^2=0$ and ${\rm \bf d}$ induces a cochain complex:
$$
C^\infty(M)\otimes \mathfrak{g}\stackrel{\rm \bf d}{\longrightarrow} \Omega^1(M)\otimes \mathfrak{g}\stackrel{\rm \bf d}{\longrightarrow}\Omega^2(M)\otimes\mathfrak{g}\stackrel{\rm \bf d}{\longrightarrow}\ldots 
$$
where $\Omega^k(M)\otimes\mathfrak{g}$, with $k\in \mathbb{Z}$, is the space of $\mathfrak{g}$-valued differential $k$-forms on $M$.  

Similarly to the case of de Rham cohomology, $\theta\in \Omega(M)\otimes\mathfrak{g}$ is said to be {\rm\bf d}-{\it closed} if ${\rm \bf d}\theta=0$ and {\rm\bf d}-{\it  exact} when $\theta={\rm \bf d}\vartheta$ for a $\vartheta\in \Omega(M)\otimes \mathfrak{g}$. Let $Z_{\rm dR}^k(M,\mathfrak{g})$ and $B_{\rm dR}^k(M,\mathfrak{g})$ be the spaces of {\rm \bf d}-closed and {\rm \bf d}-exact  $\mathfrak{g}$-valued differential $k$-forms, respectively. From ${\rm \bf d}^2=0$ it follows that $B_{\rm dR}^k(M,\mathfrak{g})\subset Z^k_{\rm dR}(M,\mathfrak{g})$, and it makes sense to define 
$$
{\rm H}_{\rm dR}^k(M,\mathfrak{g}):=\frac{Z_{\rm dR}^k(M,\mathfrak{g})}{B_{\rm dR}^k(M,\mathfrak{g})},\qquad k\in \overline{\mathbb{N}}.
$$
Since a $\mathfrak{g}$-valued differential form on $M$ is nothing but a family of $\dim \mathfrak{g}$  mutually independent differential forms on $M$, the Poincar\'e Lemma for standard differential forms can be straightforwardly extended to ${\bf d}$-closed $\mathfrak{g}$-valued differential forms on $M$. Moreover, ${\rm H}_{\rm dR}^k(M,\mathfrak{g})\simeq \bigoplus_{j=1}^{\dim\mathfrak{g}}{\rm H}_{\rm dR}^k(M)$ for $k\in \overline{\mathbb{N}}$.

Let us define a new cohomology on $\Omega(M)\otimes\mathfrak{g}$ by using the bracket $[\cdot\wedge\cdot]$ of $\mathfrak{g}$-valued differential forms on $M$ \cite{KN96}. This bracket is defined in such a way  that
$$
[\vartheta_1\otimes v_1\wedge\vartheta_2\otimes v_2]:=\vartheta_1\wedge \vartheta_2\otimes[v_1,v_2], \qquad \forall \vartheta_1\otimes v_1, \vartheta_2\otimes v_2\in \Omega(M)\otimes \mathfrak{g}, 
$$
and its value is extended by $C^\infty(M)$-bilinearity over the whole $\Omega(M)\otimes \mathfrak{g}$. This prompts us to define a new cochain complex on $\Omega(M)\otimes \mathfrak{g}$ given in the following theorem. 

\begin{theorem}\label{Coh1} Assume $\omega\in \Omega^1(M)\otimes \mathfrak{g}$ and define 
$$
{\rm \bf d}_\omega\vartheta:={\rm \bf d}\vartheta-\frac12[\omega\wedge\vartheta ],\qquad \forall\vartheta\in \Omega(M)\otimes \mathfrak{g} .
$$
If  ${\rm \bf d}_{\omega}\omega=0$, then ${\rm \bf d}_{2\omega}^2=0$ and ${\rm \bf d}_{2\omega}$ induces a cochain complex
\begin{equation}\label{chain}
C^\infty(M)\otimes \mathfrak{g}\stackrel{{\rm\bf d}_{2\omega}}{\longrightarrow} \Omega^1(M)\otimes \mathfrak{g}\stackrel{{\rm\bf d}_{2\omega}}{\longrightarrow}\Omega^2(M)\otimes\mathfrak{g}\stackrel{{\rm\bf d}_{2\omega}}{\longrightarrow}\ldots 
\end{equation}
\end{theorem}
\begin{proof} Assume that $\vartheta\in \Omega^k(M)\otimes \mathfrak{g}$ with $k\in \overline{\mathbb{N}}$. Since $\omega\in \Omega^1(M)\otimes \mathfrak{g}$ in view of the definition of ${\rm \bf d}_{2\omega}$, and the Lie bracket on $\mathfrak{g}$-valued differential forms, it follows that ${\rm\bf d}\vartheta-[\omega\wedge \vartheta]$ is a $\mathfrak{g}$-valued differential $(k+1)$-form. Hence,  the sequence of linear morphisms (\ref{chain}) is well defined.

Let us prove that ${\rm\bf d}^2_{2\omega}=0$. By evaluating ${\rm\bf d}^2_{2\omega}$ on $\vartheta$, we obtain
 \begin{equation}\label{d2t}
\begin{aligned}
{\rm \bf d}^2_{2\omega} \vartheta=-[{\bf d}\omega\wedge \vartheta]+[\omega \wedge {\bf d}\vartheta]-[\omega\wedge ({\bf d}\vartheta-[\omega\wedge \vartheta])]=-[{\bf d}\omega \wedge \vartheta]+[\omega \wedge [\omega\wedge \vartheta]].
\end{aligned}
\end{equation}
Since ${\rm\bf d}\omega=\frac 12[\omega\wedge\omega]$ by assumption, it follows that
\begin{equation}\label{3}
{\rm \bf d}^2_{2\omega} \vartheta=-\frac12[[\omega \wedge\omega]\wedge \vartheta]+[\omega \wedge [\omega\wedge \vartheta]].
\end{equation}
Let us set ${\rm d}x^J:={\rm d}x^{j_1}\wedge\ldots\wedge {\rm d}x^{j_{|J|}}$.
In local coordinates $\omega={\rm d}x^\alpha\otimes v_\alpha$ and $\vartheta={\rm d}x^J\otimes v_J$ for certain vectors $v_\alpha, v_J\in \mathfrak{g}$. Hence,
$$
{\rm\bf d}^2_{2\omega} \vartheta={\rm d}x^\alpha\wedge {\rm d}x^\beta\wedge {\rm d}x^J\otimes \frac12\left(-[[v_\alpha\wedge v_\beta ]\wedge v_J]+[v_\alpha\wedge [v_\beta \wedge  v_J]]+[v_\alpha \wedge [v_\beta\wedge v_J]]\right).
$$
The Jacobi identity for the Lie bracket in $\mathfrak{g}$ implies that
$$
{\rm\bf d}^2_{2\omega} \vartheta={\rm d}x^\alpha\wedge {\rm d}x^\beta\wedge {\rm d}x^J\otimes \frac12\left([v_\beta,[v_\alpha,\wedge v_J]+[v_\alpha,[v_\beta,v_J]]\right).
$$
Since the coefficients relative to the indices $\alpha,\beta,J$ of the above $\mathfrak{g}$-valued differential $(k+1)$-form are symmetric relative to the interchange of $\alpha$ and $\beta$, it follows that ${\rm \bf d}^2_{2\omega} \vartheta=0$. As this result remains true for any $\vartheta$ and any coordinated open subset of $M$, it turns out that ${\rm \bf d}^2_{2\omega}=0$.
\end{proof}

As in the de Rham cohomology induced by ${\bf d}$ on  $\mathfrak{g}$-valued differential forms, $\theta\in \Omega(M)\otimes\mathfrak{g}$ is said to be ${\rm \bf d}_{2\omega}$-{\it closed} if ${\rm \bf d}_{2\omega}\theta=0$ and ${\rm \bf d}_{2\omega}$-{\it exact} when $\theta={\rm \bf d}_{2\omega}\vartheta$ for a certain $\vartheta\in \Omega(M)\otimes \mathfrak{g}$. Let $Z^k_{2\omega}(M,\mathfrak{g})$ and $B^k_{2\omega}(M,\mathfrak{g})$ be the spaces of ${\rm \bf d}_{2\omega}$-closed and ${\rm \bf d}_{2\omega}$-exact $\mathfrak{g}$-valued differential $k$-forms, respectively. Since ${\rm \bf d}_{2\omega}^2=0$, it follows that $B^k_{2\omega}(M,\mathfrak{g})\subset Z_{2\omega}^k(M,\mathfrak{g})$ and we can define
$$
{\rm H}_{2\omega}^k(M,\mathfrak{g}):=\frac{Z^k_{2\omega}(M,\mathfrak{g})}{B_{2\omega}^k(M,\mathfrak{g})},\qquad k\in \overline{\mathbb{N}}.
$$
The Poincar\'e Lemma for standard differential forms can be straightforwardly extended to $\mathfrak{g}$-valued differential one-forms on $M$ relative to ${\bf d}_\omega$ and ${\bf d}_{2\omega}$. The first generalization is immediate and, in particular, ${\rm H}_{2\omega}^k(M,\mathfrak{g})\simeq \oplus_{j=1}^{\dim\mathfrak{g}}{\rm H}^k_{\rm dR}(M)$ for $k\in \overline{\mathbb{N}}$. Meanwhile,  a generalization of the Poincar\'e Lemma for the cohomology ${\bf d}_{2\omega}$ is given by the following theorem.

\begin{theorem}\label{CEWitten} Every ${\bf d}_{2\omega}$-closed differential one-form $\vartheta\in \Omega^1(M)\otimes \mathfrak{g}$ is locally ${\bf d}_{2\omega}$-exact, namely for every $p\in M$ there exists and open $U\ni p$ and $\theta\in C^\infty(U)\otimes\mathfrak{g}$ such that ${\bf d}_{2\omega}\theta=\vartheta $ on $U$. 
\end{theorem}
\begin{proof}	
	The condition ${\rm \bf d}_{2\omega}\theta=\vartheta$, with $\theta\in C^\infty(M)\otimes \mathfrak{g}$, amounts to
$$
{\rm \bf d}\theta=[\omega\wedge \theta]+\vartheta.
$$
By using the zero-curvature condition (ZCC) for this system of first-order PDEs in the unknown $\theta$, it follows that the system has a local solution if and only if ${\rm \bf d}([\omega\wedge \theta]+\vartheta)=0$. Since ${\rm \bf d}_{\omega}\omega=0$, it follows that 
$$
{\rm \bf d}([\omega\wedge\theta]+\vartheta)=[{\rm \bf d}\omega\wedge \theta]-[\omega\wedge{\bf d}\theta]+{\rm \bf d}\vartheta=\frac 12[[\omega\wedge \omega]\wedge\theta]-[\omega\wedge([\omega \wedge \theta]+\vartheta)]+{\rm\bf d}\vartheta={\bf d}_{2\omega}\vartheta.
$$
Hence, ${ \bf d}_{2\omega}\theta=\vartheta$ for a certain locally defined $\theta$  if and only if ${\bf d}_{2\omega}\vartheta=0$.
\end{proof}

It is relevant that the ${\bf d}$-closedness of certain elements of $\Omega(M)\otimes\mathfrak{g}$ is related to the existence of solutions to certain systems of partial differential equations. For instance, if ${\bf d}_\omega\omega=0$, then
$$
\forall p\in M, \exists U\ni p,\exists F\in C^\infty(U)\otimes \mathfrak{g},\,\, {\bf d}F=\Upsilon\,\,\Longleftrightarrow \,\, \Upsilon\in Z_{\rm dR}^1(M,\mathfrak{g}),
$$
$$
\forall p\in M, \exists U\ni p,\exists g\in C^\infty(U)\otimes G,\,\, {\bf d}g=R_g\omega,\,\,\omega\in C^\infty(U)\otimes\mathfrak{g}\,\,\Longleftrightarrow \,\, \omega\in Z_{2\omega}^1(M,\mathfrak{g}),
$$
$$
\forall p\in M, \exists U\ni p,\exists F\in C^\infty(U)\otimes \mathfrak{g},\,\, {\bf d}F-[\omega,F]=\Upsilon\,\,\Longleftrightarrow \,\, \Upsilon\in Z_{2\omega}^1(M,\mathfrak{g}).
$$
The first and third cases are immediate. The second one follows from the fact that $G$ and $\mathfrak{g}$ are matrix Lie groups and Lie algebras, respectively, and ${\bf d}$ acts on them as a de Rham differential on their matrix entries.

The above results can be generalized to the hereafter called parametrized  $\mathfrak{g}$-valued differential forms. 
These differential forms will appear naturally in the study of Lax pairs and immersion formulas. We define $\mathcal{N}:=J^p\times \Lambda$, where $\Lambda$ is a one-dimensional submanifold of $\mathbb{C}$  coordinated by the variable $\lambda$.  

\begin{definition} A {\it parametrized $\mathfrak{g}$-valued differential form} is a family of $\mathfrak{g}$-valued differential forms on $M$, say
\begin{equation}\label{form}
\chi(u(x),\lambda)=\chi_J(u(x),\lambda){\rm d}x^J,\qquad \chi_J(u(x),\lambda)\in C^\infty(M),
\end{equation}
parametrized by arbitrary sections $u\in \Gamma(\pi)$ and $\lambda\in \Lambda$. A  parametrized $\mathfrak{g}$-valued differential form {\it on-shell} is a family of $\mathfrak{g}$-valued differential forms  (\ref{form}) where  $u\in \Gamma(\pi)$ is a solution of $\Delta(j^p_xu)=0$. We write $\Omega_\mathcal{N}(M)\otimes \mathfrak{g}$  (resp. $\Omega_\mathcal{S}(M)\otimes \mathfrak{g}$) for the space of  parametrized  $\mathfrak{g}$-valued differential forms (resp. on-shell).
\end{definition}

The space $\Omega_\mathcal{N}(M)\otimes \mathfrak{g}$ admits a standard grading $\Omega_\mathcal{N}(M)\otimes \mathfrak{g}=\bigoplus_{k\in \mathbb{Z}}\Omega^k_\mathcal{N}(M)\otimes \mathfrak{g}$ compatible with the exterior product of parametrized $\mathfrak{g}$-valued differential forms that is defined in the natural way. A similar grading can be applied to parametrized $\mathfrak{g}$-valued differential equations on-shell.
The elements of each $\Omega^k_\mathcal{N}(M)\otimes \mathfrak{g}$ (resp.  $\Omega^k_\mathcal{S}(M)\otimes \mathfrak{g}$) are called  {\it parametrized $\mathfrak{g}$-valued differential $k$-forms} (resp. {\it on-shell}). 

 The following theorem allows us to define two cohomologies on $\Omega_\mathcal{S}(M)\otimes \mathfrak{g}$. It is indeed a natural generalization of Theorem \ref{Coh1} to $\Omega_{\mathcal{N}}(M)\otimes\mathfrak{g}$ and $\Omega_{\mathcal{S}}(M)\otimes\mathfrak{g}$.
 
\begin{theorem}\label{Poincare1} Let $\omega\in \Omega^1_{\mathcal{N}}(M)\otimes \mathfrak{g}$ be such that $
	{\rm \bf d}_\omega\omega=0
	$ only on particular solutions of (\ref{PDE}). Let ${\rm \bf d}$ and ${\rm \bf d}_{2\omega}$ be the linear operators on $\Omega_\mathcal{N}(M)\otimes \mathfrak{g}$ acting on elements $\vartheta=\vartheta_J{\rm d}x^J$ in the form
\begin{equation}\label{ex1}
{\rm \bf d}\vartheta:=(d_\alpha\vartheta_J){\rm d}x^{\alpha}\wedge {\rm d}x^J,\qquad {\rm \bf d}_{2\omega}\vartheta:={\rm \bf d}\vartheta-[\omega\wedge\vartheta],
\end{equation}
where
$
(d_\alpha\vartheta_J)(u(x),\lambda):=\partial_\alpha [\vartheta_J(u(x),\lambda) 
]$ and $\omega\in \Omega^1_\mathcal{N}(M)\otimes \mathfrak{g}$. The space of $\Omega_\mathcal{N}(M)\otimes \mathfrak{g}$ admits a cohomology induced by ${\rm \bf d}$. Meanwhile, the space of $\Omega_\mathcal{S}(M)\otimes \mathfrak{g}$ admits two cohomologies given by ${\rm \bf d}$ and ${\rm \bf d}_{2\omega}$.
\end{theorem}
\begin{proof} In view of (\ref{ex1}), the differential ${\rm \bf d}$ of a parametrized  $\mathfrak{g}$-valued differential $k$-form is a parametrized  $\mathfrak{g}$-valued differential $(k+1)$-form. Additionally, the commutativity of partial derivatives $\partial_\alpha,\partial_\beta$, for $\alpha,\beta\in\overline{1,m}$, and its action on the differential functions $\vartheta_J(u(x),\lambda)$ for each fixed $u(x)$ implies that $[d_\alpha,d_\beta]=0$. Hence,
$$
{\rm \bf d}^2\vartheta=(d_\beta d_\alpha\vartheta_J)(u(x),\lambda){\rm d}x^\beta\wedge {\rm d}x^\alpha\wedge dx^J=0
$$
because the coefficients of ${\bf d}^2\vartheta$, i.e. $d_\beta d_\alpha\vartheta_J$, are symmetric relative to the interchange of $\alpha$ and $\beta$. Additionally,
$$
{\rm \bf d} [\omega_1\wedge \omega_2]=[{\bf d} \omega_1\wedge \omega_2]+(-1)^s[\omega_1\wedge {\bf d} \omega_2],\quad\forall \omega_1\in \Omega_\mathcal{N}^s(M)\otimes \mathfrak{g},\quad \forall\omega_2\in \Omega_\mathcal{N}(M)\otimes \mathfrak{g}.
$$
Following the same ideas of Theorem \ref{Coh1} and using the previous fact, we see  that ${\rm \bf d}_{2\omega}$ induces a cohomology on $\Omega_\mathcal{S}(M)\otimes \mathfrak{g}$.
\end{proof}

The above theorem denotes the different differentials for the parametrized  $\mathfrak{g}$-valued differential forms in the same way as the differentials for standard $\mathfrak{g}$-valued differential forms. This simplifies the notation and it is not misleading since the meaning of ${\bf d}$ and ${\bf  d}_{2\omega}$ is clear taking into account on which kind of $\mathfrak{g}$-valued differential form the operator is acting on. Moreover, both operators are essentially the same thing, but they operate in different spaces. Similarly to previous cohomologies, one can define closeness and exactness relative to the operators ${\bf d}$ and ${\bf d}_\omega$. The spaces $Z^k_{\mathcal{S}}(M,\mathfrak{g})$, $B^k_{\mathcal{S}}(M,\mathfrak{g})$, and ${\rm H}^k_{\mathcal{S}}(M,\mathfrak{g})$ of closed, open, and equivalence classes of parametrized $\mathfrak{g}$-valued differential forms on-shell modulo exact ones can be defined as for previous cohomologies. Similarly, one can define the corresponding spaces for parametrized $\mathfrak{g}$-valued differential forms. Moreover, applying Theorem \ref{CEWitten} for each fixed $u(x)$ and $\lambda$, we get the following corollary.

\begin{corollary}\label{DoPragi} Every ${\bf d}_{2\omega}$-closed element of $\Omega^1_\mathcal{S}(M)\otimes \mathfrak{g}$ is locally ${\bf d}_{2\omega}$-exact.
\end{corollary}

More easily, one has the following trivial result.

\begin{corollary}\label{DoPragi2} Every ${\bf d}$-closed element of $\Omega^k_\mathcal{N}(M)\otimes \mathfrak{g}$ or $\Omega^k_\mathcal{S}(M)\otimes\mathfrak{g}$ is locally ${\bf d}$-exact. 
\end{corollary}

\section{Spectral differential forms}

It turns out that parametrized $\mathfrak{g}$-valued differential forms of the previous section can be generated, in relevant cases, by means of a new type of $\mathfrak{g}$-valued differential forms on spaces of infinite jets: the hereafter called spectral differential forms. Although the theoretical description of these new structures is more complicated than the description of parametrized $\mathfrak{g}$-valued differential forms, they offer numerous practical advantages in calculations, which justifies their introduction. In particular, this formalism permits us to use the powerful machinery of the theory of jet bundles to obtain Lie symmetries, immersion formulas, and other related structures. Relevantly, this formalism will fill some theoretical details in the jet formalism of the immersion formulas given in \cite{Grundland15}. 

\begin{definition} A {\it spectral differential form} is a family of $\mathfrak{g}$-valued differential forms on $J^\infty$ parametrized by a spectral parameter $\lambda\in\Lambda\subset \mathbb{C}$ for a certain submanifold $\Lambda\subset \mathbb{C}$ and taking the form
	\begin{equation}\label{NIF}
	{\bf \omega_\lambda}:=\omega_J([u],\lambda){\rm d}x^J,
	\end{equation}
	where $\omega_J([u],\lambda)$  is an arbitrary function in $C_\lambda^\infty(J^\infty):=\cup_{p\in \bar{\mathbb{N}}}C^\infty(J^p\times \Lambda)$. An {\it spectral differential form on-shell} is the restriction of a spectral differential form (\ref{NIF}) to a $\lambda$-parametrized family of $\mathfrak{g}$-valued differential forms on $\mathcal{E}^\infty$. We write $\Omega_\lambda(J^\infty)\otimes \mathfrak{g}$  (resp. $\Omega_\lambda(\mathcal{E}^\infty)\otimes \mathfrak{g}$) for the space of  spectral differential forms (resp. on-shell).
\end{definition}

The space $\Omega_\lambda(J^\infty)\otimes \mathfrak{g}$ admits a standard grading $\Omega_\lambda(J^\infty)\otimes \mathfrak{g}=\bigoplus_{k\in \mathbb{Z}}\Omega_\lambda^k(J^\infty)\otimes \mathfrak{g}$ compatible with the exterior product of spectral differential forms defined in the natural way. A similar grading can be applied to spectral differential forms on-shell.
The elements of each $\Omega^k_\lambda(J^\infty)\otimes \mathfrak{g}$ (resp.  $\Omega^k_\lambda(\mathcal{E}^\infty)\otimes \mathfrak{g}$) are called  {\it spectral differential $k$-forms} (resp. {\it on-shell}). 

Every {\it spectral differential form} induces a unique parametrized $\mathfrak{g}$-valued differential form. 
Indeed, given (\ref{NIF}) and a section ${s}\in \Gamma(\pi)$ of the form $s(x)=(x,u(x))$, a parametrized $\mathfrak{g}$-valued differential form  is defined by 
$
\omega(u(x),\lambda):=j^\infty s^*\omega_\lambda.
$
If $u(x)$ is a solution of our system of PDEs, then $\omega(u(x),\lambda)$ becomes a parametrized $\mathfrak{g}$-valued differential form on-shell. Meanwhile, a parametrized $\mathfrak{g}$-valued differential form need not come from a spectral differential form.

Since spectral differential forms give rise to parametrized $\mathfrak{g}$-valued differential forms, it becomes relevant how to apply the formalism for parametrized $\mathfrak{g}$-valued differential forms directly to spectral  differential forms. This is accomplished by the following theorem, which provides an extension of the horizontal differential on the infinite-dimensional jet bundle $J^\infty$ (see \cite{KV11}).

\begin{theorem}\label{Poincare1p} Let $\omega\in \Omega^1_\lambda(J^\infty)\otimes\mathfrak{g}$ be such that ${\bf d}_\omega\omega=0$ on $\mathcal{E}^\infty$. Let ${\rm \bf d}$ and ${\rm \bf d}_{2\omega}$ be the linear operators on $\Omega_\lambda(J^\infty)\otimes \mathfrak{g}$ acting on homogeneous elements $\vartheta=\vartheta_J{\rm d}x^J\in \Omega^k_\lambda(J^\infty)\otimes \mathfrak{g}$ given by
	\begin{equation}
	\label{def}
	{\rm \bf d}\vartheta:=(D_\alpha\vartheta_J){\rm d}x^{\alpha}\wedge {\rm d}x^J,\qquad {\rm \bf d}_{2\omega}\vartheta:={\rm \bf d}\vartheta-[\omega\wedge\vartheta].
	\end{equation}
	The space $\Omega_\lambda(J^\infty)\otimes\mathfrak{g}$ admits a cohomology given by ${\bf d}$. The space $\Omega_\lambda(\mathcal{E}^\infty)\otimes \mathfrak{g}$ admits two cohomologies given by ${\rm \bf d}$ and ${\rm \bf d}_{2\omega}$.
\end{theorem}
\begin{proof} The formula (\ref{def}) ensures that the differential ${\rm \bf d}$ of a spectral differential $k$-form on $J^\infty$ is a spectral differential $(k+1)$-form. Since $[D_\alpha,D_\beta]=0$ on $J^\infty$ (cf. \cite{Ol93}),
	$$
	{\rm \bf d}^2\vartheta=(D_\beta D_\alpha\vartheta_J){\rm d}x^\beta\wedge {\rm d}x^\alpha\wedge dx^J=0
	$$
	because the coefficients of this spectral differential  $(k+2)$-form, namely $D_\beta D_\alpha\vartheta_J$, are symmetric relative to the interchange of $\alpha$ and $\beta$. Hence, ${\bf d}$ gives rise to a cohomology in $\Omega_\lambda(J^\infty)\otimes \mathfrak{g}$. Meanwhile,
	$$
	{\rm \bf d} [\omega_1\wedge \omega_2]=[{\bf d} \omega_1\wedge \omega_2]+(-1)^s[\omega_1\wedge {\bf d} \omega_2],\quad\forall \omega_1\in \Omega_\lambda^s(J^\infty)\otimes \mathfrak{g},\quad \forall\omega_2\in \Omega_\lambda(J^\infty)\otimes \mathfrak{g}.
	$$
	
  The vector fields $D_\alpha$ are tangent to $\mathcal{E}^\infty$, which allows us to define the action of $D_\alpha$ on $C_\lambda^\infty(\mathcal{E}^\infty)$ and to restrict  ${\bf d}$  and ${\bf d}_{2\omega}$ to spectral differential forms on-shell. Moreover, since $[D_\alpha,D_\beta]=0$ on $\mathcal{E}^\infty$, it follows that ${\bf d}^2=0$, which gives rise to a differential and a cochain simplex on  $\Omega_\lambda(\mathcal{E}^\infty)\otimes\mathfrak{g}$. This along with the definition (\ref{def}) allows us to apply the ideas of the proof in Theorem \ref{Coh1} so as to prove that ${\bf d}_{2\omega}$ gives rise to a cohomology in $\Omega_\lambda(\mathcal{E}^\infty)\otimes \mathfrak{g}$. 
\end{proof}

Let $H_{\rm dR}^k(J^\infty,\mathfrak{g})$ and $H_{\rm dR}^k(\mathcal{E}^\infty,\mathfrak{g})$ be the cohomologic groups induced by ${\bf d}$ on $\Omega_\lambda(J^\infty)\otimes\mathfrak{g}$ and $\Omega_\lambda(\mathcal{E}^\infty)\otimes\mathfrak{g}$. The proof of Theorem \ref{Poincare1p} and the form of ${\bf d}$ show that $H_{\rm dR}^{k}(J^\infty,\mathfrak{g})\simeq \bigoplus_{j=1}^{\dim \mathfrak{g}}H_{\rm dR}^k(J^\infty)$. Since $M\times N$ is a strong deformation retract of $J^\infty$, then $H_{\rm dR}^k(J^\infty)\simeq H_{\rm dR}^k(M\times N)$ (see \cite{GMS00} for details).

It is worth commenting on the space on which the spectral differential form $\omega$ vanishes. The LSP for a system of PDEs gives rise to a parametrized $\mathfrak{g}$-valued differential form $\omega$ such that ${\bf d}_\omega\omega=0$ for every $\lambda\in \Lambda$ only on particular solutions of our system of PDEs.  In practical applications $\omega$ is then described by a spectral differential one-form, $\omega_\lambda$, vanishing for every $\lambda\in \Lambda$ on those points of $J^\infty$ projecting onto $\mathcal{E}$. Then, the submanifold $\mathcal{E}^\infty$ is contained in this space and ${\bf d}_{\omega_\lambda}\omega_\lambda=0$ on $\mathcal{E}^\infty$.

	Let us prove the following Poincar\'e Lemma-type result, which is an extension of the Poincar\'e Lemma for the cohomology of horizontal forms on diffieties and infinite-dimensional jet bundles $J^\infty$ \cite{KV99}.
	
	\begin{theorem} Let $J^0$ and $\mathcal{E}^\infty$ be homotopic to $\mathbb{R}^m$. Then every ${\bf d}$-closed spectral differential form or spectral differential form on-shell is exact. Every ${\bf d}_{2\omega}$-closed spectral differential one-form on-shell is locally exact.
		\end{theorem}
	\begin{proof} The differential operators ${\bf d}$ and ${\bf d}_{2\omega}$  induce cochain complexes with respect to the gradings $\Omega_\lambda(J^\infty)\otimes \mathfrak{g}\simeq\oplus_{k\in \mathbb{Z}} \Omega^k_\lambda(J^\infty)\otimes \mathfrak{g}$ and $\Omega_\lambda(\mathcal{E}^\infty)\otimes \mathfrak{g}\simeq\oplus_{k\in \mathbb{Z}} \Omega^k_\lambda(\mathcal{E}^\infty)\otimes \mathfrak{g}$, i.e. we have the cochain complexes

\begin{equation}\label{chain2}
C_\lambda^\infty(J^\infty)\otimes \mathfrak{g}\stackrel{{\rm\bf d}}{\longrightarrow} \Omega_\lambda^1(J^\infty)\otimes \mathfrak{g}\stackrel{{\rm\bf d}}{\longrightarrow}\Omega_\lambda^2(J^\infty)\otimes\mathfrak{g}\stackrel{{\rm\bf d}}{\longrightarrow}\ldots 
\end{equation}	
\begin{equation}\label{chain3}
	C_\lambda^\infty(\mathcal{E}^\infty)\otimes \mathfrak{g}\stackrel{{\rm\bf d},{\rm\bf d}_{2\omega}}{\longrightarrow} \Omega_\lambda^1(\mathcal{E}^\infty)\otimes \mathfrak{g}\stackrel{{\rm\bf d},{\rm\bf d}_{2\omega}}{\longrightarrow}\Omega_\lambda^2(\mathcal{E}^\infty)\otimes\mathfrak{g}\stackrel{{\rm\bf d},{\rm\bf d}_{2\omega}}{\longrightarrow}\ldots 
\end{equation}	
For every fixed $\lambda$, the space $\Omega^k_\lambda(J^\infty)\otimes\mathfrak{g}$ is isomorphic to $\dim\mathfrak{g}$-copies of the so-called {\it horizontal space} 
$$
\bar{\Lambda}^k(J^\infty):=\{\vartheta:=\vartheta_J([u]){\rm d}x^J:[u]\in J^\infty,\vartheta_J([u])\in \mathfrak{g}\}$$ of  horizontal differential forms on $J^\infty$. Moreover, these spaces give rise to a cochain complex for the so-called {\it horizontal differential} $d_h$ of the {\it variational complex} (see \cite[Section 5.4]{Ol93}), which coincides with ${\bf d}$ when $\dim\mathfrak{g}=1$. Additionally ${\bf d}={\rm d}_h\oplus\ldots\oplus {\rm d}_h$ ($\dim\mathfrak{g}$ times). If a horizontal form is closed relative to $d_h$, then it is locally exact (cf. \cite[Theorem 5.82]{Ol93}). Hence,  a ${\bf d}$-closed spectral differential form will be always locally  ${\bf d}$-exact. 

The previous result can be extended to spectral differential forms on-shell by using the diffiety $(\mathcal{E}^\infty,\mathcal{C}|_{ \mathcal{E}^\infty})$. In fact, it is useful to note that, for a trivial differential equation $\mathcal{E}=J^p$, one has that $(\mathcal{E}^\infty,\mathcal{C}|_{\mathcal{E}^\infty})=(J^\infty,\mathcal{C})$ and the following commentaries retrieve the results of the previous paragraph. The space of spectral differential forms on-shell for a particular $\lambda$ is isomorphic to $\dim \mathfrak{g}$-copies of the space of horizontal forms of the diffiety, $\Lambda_h(\mathcal{E}^\infty)$, and the restriction of the differential ${\bf d}$ to $\Omega_\lambda(\mathcal{E}^\infty)\otimes\mathfrak{g}$ is indeed the horizontal differential on the diffiety $\mathcal{E}^\infty$ (see \cite{Vi01} for details on this and following commentaries). As every closed horizontal differential form is locally closed relative to the differential of the diffiety, it follows that every ${\bf d}$-closed form on $\mathcal{E}^\infty$ is locally exact. Since $\mathcal{E}^\infty$ is homotopic to $\mathbb{R}^m$, every ${\bf d}$-closed spectral differential form on-shell is exact.
			
		The proof for the local exactness of ${\bf d}_{2\omega}$-closed spectral differential one-forms follows from writing the partial differential equations $\Upsilon={\bf d}_{2\omega}F$, for $\Upsilon\in \Omega^1_\lambda(\mathcal{E}^\infty)\otimes\mathfrak{g}$ in terms of ${\bf d}$. The obtained differential equation in $F$ is integrable if the system of PDEs on $\mathcal{E}^\infty$ is analytic and formally integrable, which amounts to the fact that ${\bf d}_{2\omega}\Upsilon=0$ (cf. \cite{KL08}). Hence, a local solution $F$ can be obtained and $\Upsilon$ becomes locally exact.
\end{proof}

Previous theorems denote the differentials for the parametrized  spectral differential forms in the same way as those for standard parametrized and standard $\mathfrak{g}$-valued differential forms. This simplifies the notation and it does not lead to a mistake: the meaning of ${\rm \bf d}$ and ${\rm\bf  d}_{2\omega}$ is clear taking into account which kind of $\mathfrak{g}$-valued differential form the operator is acting on. Moreover, both operators are essentially the same thing, but they operate in different spaces.

\section{LSPs and $\mathfrak{g}$-valued differential forms}
To analyse immersion formulas for soliton surfaces in Lie algebras and to generalize them to immersed soliton submanifolds in Lie algebras, we will  describe integrability conditions in terms of cohomological techniques. 

Let us consider an integrable (in the sense of having a LSP) system of PDEs on $J^p$  of the form
\begin{equation}
\Delta(j^p_xu)=0,\label{2.1}
\end{equation}
for a function $\Delta:J^p\rightarrow \mathbb{R}^s$. As in Section \ref{ClaLie}, the above system is assumed to have maximal rank and to be locally solvable. Recall that the maximal rank assumption implies that the zeros of $\Delta$ give rise to a submanifold $\mathcal{E}\subset J^p$. 

Let $G$ be a matrix Lie group and let $\mathfrak{g}$ be its matrix Lie algebra.
Consider that the LSP problem related to (\ref{2.1}) is given by a $(\lambda,u(x))$-parametrized family of systems of linear PDEs 
\begin{equation}\label{LSP1}
\partial_\alpha\Phi=U_\alpha({ u}(x),\lambda)\Phi,\qquad  \alpha\in \overline{1,m},\quad \Phi\in G,\quad\,U_\alpha(u(x),\lambda)\in \mathfrak{g},\quad \lambda \in \Lambda\subset \mathbb{C},
\end{equation}
whose integrability condition, the {\it Zero-Curvature Condition} (ZCC), i.e. 
\begin{equation}\label{sur1}
\partial_\beta U_\alpha-\partial_\alpha U_\beta+[U_\alpha,U_\beta]=0,\qquad \forall \lambda\in \Lambda,\qquad \forall \alpha,\beta \in \overline{1,m}, 
\end{equation}
is satisfied if and only if $u(x)$ is a solution to (\ref{2.1}). Let us describe this result cohomologically. In this respect, we define a parametrized $\mathfrak{g}$-valued differential one-form $\omega({ u}(x),\lambda):=U_\alpha({ u}(x),\lambda){\rm d}x^\alpha$. 
 Then, the {\it Zero-Curvature Condition} (ZCC) for (\ref{2.1}) amounts to looking for the curves $u(x):M\rightarrow N$ ensuring that  
\begin{equation}\label{con1}
{\rm \bf d}_{\omega}{\omega}:={\rm \bf d}\omega-\frac12[\omega\wedge\omega]=0,\qquad \forall \lambda\in \Lambda.
\end{equation} 
This equivalence appears as a direct consequence of writing (\ref{con1}) in coordinates, namely
$$
0\!=\!\partial_\beta U_\alpha{\rm d}x^\beta\wedge {\rm d}x^\alpha -\frac {\left[U_\alpha {\rm d}x^\alpha\wedge U_\beta {\rm d}x^\beta\right]}2\!\!=\!\!\sum_{\alpha<\beta}\left(\partial_\beta U_\alpha\!-\!\partial_\alpha U_\beta  \!+\![U_\alpha,U_\beta] \right){\rm d}x^\beta\wedge {\rm d}x^\alpha,
$$
for all $\alpha,\beta\in \overline{1,m}$ and $\lambda\!\in\! \Lambda$.
Hence, the above expression vanishes for a certain $u(x)$ if and only if  (\ref{sur1}) holds.

\begin{example} The complex projective space $\mathbb{C}P^1$ can be understood via the {\it Hopf fibration} $h:S^3\rightarrow \mathbb{C}P^1\simeq S^3/S^1\simeq S^2$ as the space of orbits in the three-dimensional unit sphere $S^3\subset \mathbb{C}^2$ (relative to the canonical Hermitian product in $\mathbb{C}^2$) with respect to the natural action of $S^1$, understood as the Lie group of complex numbers with module one, on $S^3$ (see \cite{Ho31,Ur03} for details). Equivalently, $\mathbb{C}P^1$ amounts to the space of rank-one Hermitian projectors $P$ onto $\mathbb{C}^2$, namely $P^\dagger=P$, $P^2=P$ and ${\rm tr} P=1$ \cite{Ei78,GP78}. The solutions for the $\mathbb{C}P^1$ sigma model are given by mappings $P:S^2\rightarrow \mathbb{C}P^1$ where $S^2$ is understood as the compactification of $\mathbb{C}$ by gluing the infinity point  topologically. The differential equations for the $\mathbb{C}P^1$ sigma model are given by 
	\begin{equation}\label{CP2}
	[\partial_+\partial_-P(x,y),P(x,y)]=\frac14[\nabla^2 P,P]=0,\qquad  \nabla^2:=\partial_x^2+\partial_y^2,\quad \partial_\pm:=\frac12(\partial_x\pm {\rm i}\partial_y),
	\end{equation}
	where $[\cdot,\cdot]$ stands for the commutator of operators.
	Hence, the $\mathbb{C}P^1$ sigma model can be understood as a system of PDEs on the jet bundle $J^2$ relative to the projection onto the second factor $\pi:\mathbb{C}P^1\times S^2\rightarrow S^2$. This system admits a LSP problem with a purely imaginary spectral parameter $\lambda\in \mathbb{C}$. To clearly show the geometric properties of this LSP, it is convenient to introduce the new variable $\theta:={\rm i}(P-{\rm Id}_2/2)\in\mathfrak{su}(2)$. Then, the LSP reads
	$$\left\{\begin{aligned}
	\partial_x\Phi&=U_x(\theta(x,y),\lambda)\Phi\\\partial _y\Phi&=U_y(\theta(x,y),\lambda)\Phi\end{aligned}\right.,\qquad \Phi\in SU(2),\qquad \bar \lambda=-\lambda,
	$$
	where
	$$
	 \left\{\begin{aligned}
	 &U_x(\theta(x,y),\lambda):=\frac{-2}{1-\lambda^2}([\partial_x\theta(x,y),\theta(x,y)]-{\rm i}\lambda [\partial_y\theta(x,y),\theta(x,y)])\in \mathfrak{su}(2),\\
	 &U_y(\theta(x,y),\lambda):=\frac{-2}{1-\lambda^2}([\partial_x\theta(x,y),\theta(x,y)]{\rm i}\lambda+ [\partial_y\theta(x,y),\theta(x,y)])\in \mathfrak{su}(2).
	 \end{aligned}
	 \right.
	$$
	This gives rise to a real parametrized $\mathfrak{su}(2)$-valued differential one-form on $\mathbb{C}$ of the form 
	$$
	\omega_{CP}:=U_x(\theta(x,y),\lambda){\rm d}x+U_y(\theta(x,y),\lambda){\rm d} y.
	$$
	Observe that $\theta(x,y),\partial_x\theta(x,y),\partial_y\theta(x,y)$ are understood as particular curves in $\mathfrak{su}(2)$.
	\end{example}

The LSP of the $\mathbb{C}P^1$ model, its parametrized $\mathfrak{su}(2)$-valued differential one-form, and the corresponding system of PDEs can be understood as structures on a jet bundle. This motivates the following geometric approach to the LSP equations and related structures based upon jets and spectral differential forms taking values in $\mathfrak{su}(2)$.

The system of partial differential equations (\ref{2.1}) can be naturally extended to a manifold $\mathcal{E}^\infty\subset J^\infty$ by assuming $[u]\in J^\infty$ and adding the conditions $D_J[\nabla^2P,P]=0$ for an arbitrary multi-index $J$. This allows us to define a new submanifold $\mathcal{E}^\infty\subset J^\infty$ where all previous functions on $J^\infty$ vanishes. Then, the total derivatives $D_x,D_y$ become tangent to $\mathcal{E}^\infty$.  

Let us now provide a new formalism of the LSP problem in the language of jets. Essentially, this is a reinterpretation of  the algebraic approach given in \cite{FGFL00} in terms of functions on rings and an improvement of the method on spaces of jets detailed in \cite{Grundland15}.

Consider that the hereafter called {\it jet LSP} for the differential equation (\ref{2.1}) is a system of PDEs on $J^\infty$ of the form
\begin{equation}
D_\alpha\Phi([u],\lambda)=U_\alpha([u],\lambda)\Phi([u],\lambda),\qquad
\alpha\in \overline{1,m},\qquad \Phi([u],\lambda)\in G.\label{LSPn}
\end{equation}

The system of PDEs (\ref{LSPn}) must be involutive to be integrable, i.e. $D_\alpha D_\beta\Phi([u],\lambda)=D_\beta D_\alpha \Phi([u],\lambda)$. Nevertheless, this system of partial differential equations is such that the unknown is defined on an infinite-dimensional jet space $J^\infty$. Hence, it is not obvious that this condition is sufficient to ensure integrability. Nevertheless, it can be proved under quite general conditions, e.g. the system must be analytic \cite{KL08}, that the involutivity of the system ensures its integrability.

The following theorem enables us to extend the notion of integrability on finite-dimensional manifolds to $J^\infty$  and the diffiety $\mathcal{E}^\infty$. 
\begin{theorem} The differential equation (\ref{LSPn}) has a solution $\Phi([u],\lambda)$, with $[u]\in\mathcal{E}^\infty$, if and only if ${\bf d}_{\omega}\omega=0$ on $\mathcal{E}^\infty$. 
	\end{theorem}
\begin{proof} Let us prove the direct part. If (\ref{LSPn}) has a solution on $\mathcal{E}^\infty$, then
	$$
D_\beta D_\alpha\Phi=D_\alpha D_\beta\Phi,\quad \forall \lambda\in \Lambda, \forall \alpha,\beta\in \overline{1,m} \quad \Longrightarrow \quad D_\beta U_\alpha-D_\alpha U_\beta+[U_\alpha,U_\beta]=0,\quad \forall \lambda \in \Lambda,
	$$
	on $\mathcal{E}^\infty$. Then,
	\begin{equation}\label{eqp}
	{\bf d}_{\omega}\omega=\sum_{\alpha<\beta}(D_\beta U_\alpha-D_\alpha U_\beta+[U_\alpha,U_\beta]){\rm d}x^\alpha\wedge {\rm d}x^\beta=0,\qquad \forall \lambda\in \Lambda
	\end{equation}
	on $\mathcal{E}^\infty$, whereupon the direct part follows.  The converse is immediate by taking into account that involutivity and analiticity of the system (\ref{LSPn}) on $\mathcal{E}^\infty$ gives rise to a solution (cf. \cite{KL08}).
	
\end{proof}

It is worth noting that the pull-back by an arbitrary section $s(x):=(x,u(x))$ of $J^0$ of the structures related to spectral differential forms permit us to recover the formalism for parametrized $\mathfrak{g}$-valued differential forms. In particular, $j^\infty s^* {\bf d}_\omega\omega=0$ retrieves (\ref{sur1}). Moreover, $j^\infty s^* \Phi([u],\lambda)$ provides a solution for the pull-back of the system (\ref{LSP1}), namely
$$
j^\infty s^* D_\alpha \Phi([u],\lambda)=j^\infty s^*[U([u],\lambda)\Phi([u],\lambda)]\Longleftrightarrow \partial_\alpha \Phi(u(x),\lambda)=U_\alpha(u(x),\lambda)\Phi(u(x),\lambda).
$$
In other words, the infinite-dimensional jet approach and spectral differential forms represent a way of representing the standard formalism for LSP when the structures appearing can be defined on infinite-dimensional jet manifolds.

\begin{example}  The $\mathbb{C}P^1$ sigma model can be naturally considered as a differential equation on $J^2$ of the form
	\begin{equation}\label{LSPinf}
	[D^2_x\theta+D_y^2\theta,\theta]=0.
	\end{equation}
	This gives as a submanifold $\mathcal{E}_{sm}\subset J^2$. If the additional conditions are assumed
	$$
	D_x^pD^q_y[D^2_x\theta+D_y^2\theta,\theta]=0,\qquad \forall p,q\in \overline{\mathbb{N}},
	$$
	then a submanifold $\mathcal{E}_{sm}^\infty\subset J^\infty$ is obtained and the derivatives $D_x,D_y$ are tangent to it. To describe (\ref{LSPinf}), a LSP on $J^\infty$ is introduced:
	$$\left\{\begin{aligned}
D_x\Phi([\theta],\lambda)&=U_x([\theta],\lambda)\Phi([\theta],\lambda)\\D _y\Phi([\theta],\lambda)&=U_y([\theta],\lambda)\Phi([\theta],\lambda)\end{aligned}\right.,\qquad \Phi([\theta],\lambda)\in SU(2),\qquad \bar \lambda=-\lambda,
$$
where
$$
\left\{\begin{aligned}
&U_x([\theta],\lambda):=\frac{-2}{1-\lambda^2}([D_x\theta,\theta]-{\rm i}\lambda [D_y\theta,\theta])\in \mathfrak{su}(2),\\
&U_y([\theta],\lambda):=\frac{-2}{1-\lambda^2}([D_x\theta,\theta]{\rm i}\lambda+ [D_y\theta,\theta])\in \mathfrak{su}(2).
\end{aligned}
\right.
$$
The corresponding spectral differential one-form on $J^\infty$ is given by
	$$
\omega([\theta],\lambda):=U_x([\theta],\lambda){\rm d}x+U_y([\theta],\lambda){\rm d} y.
$$
Hence
$$
\begin{aligned}
{\bf d}_\omega\omega&=D_\beta U_\alpha\otimes {\rm d}x^\beta\wedge {\rm d}x^\alpha -\frac12\left[(U_\alpha {\rm d}x^\alpha)\wedge (U_\alpha{\rm d}x^\alpha)\right]\\
&=\left(D_y U_x-D_x U_y  +[U_x,U_y] \right){\rm d}y\wedge {\rm d}x\\
&=[D_x^2+D_y^2)\theta,\theta]{\rm d}y\wedge {\rm d}x.
\end{aligned}
$$
Composing this with the infinite-prolongation $j^\infty s$ of a section $s(x,y):=(x,y,\theta(x,y))$ related to an arbitrary $\theta(x,y)$, we obtain that $j^\infty s^*{\bf d}_\omega\omega=0$ amounts to 
$$
\begin{aligned}
0&=j^\infty s^*{\bf d}_\omega\omega\\
&=(\partial_y U_x-\partial_x U_y  +[U_x,U_y]){\rm d}y\wedge {\rm d}x\\
&=(\partial_y U_x(\theta(t),\lambda)-\partial_x U_y(\theta(t),\lambda)+[U_x(\theta(t),\lambda),U_y(\theta(t),\lambda)]){\rm d}y\wedge {\rm d}x.
\end{aligned}
$$
	\end{example}

Finally, let us consider an assumption that will be useful in the following sections and that appears in different practical applications of immersion formulas in the literature.

Subsequently, we will assume that there exists a particular on-shell solution  $\Phi([u],\lambda)$ to (\ref{LSP1}) satisfying the so-called {\it on-shell condition}
$$
\lim_{\lambda \rightarrow \infty}\Phi([u],\lambda)={\rm Id}_N\in SU(2),\qquad \forall [u]\in \mathcal{E}_{sm}^\infty,
$$
where ${\rm Id}_N$ is an $N\times N$ identity matrix. 
\begin{example} A solution to the jet LSP for the $\mathbb{C}P^1$ sigma model can be written (cf. \cite{GLM16}) as 
	$$
\Phi([\theta],\lambda)={\rm Id}_2+\left(\frac{4\lambda}{(1-\lambda)^2}-\frac{2}{1-\lambda}\right)(-{\rm i}\theta+{\rm Id}_2/2)=\frac{\lambda(\lambda+1)}{(\lambda-1)^2}{\rm Id}_2-i\frac{6\lambda-2}{(1-\lambda)^2}\theta\in SU(2),
$$
for $[u]\in \mathcal{E}^\infty$, $\lambda:={\rm i}t$, and $t\in \mathbb{R}.$
	\end{example}

\section{Immersion formulas}
Now we are interested in looking for a simultaneous infinitesimal deformation of the LSP and the zero curvature condition. This will be done in a geometric manner, which will recover the case of immersed soliton surfaces as a particular instance, and it will enable us to generalize the procedure to any immersed submanifold in Lie algebras. Our formalism will use spectral differential forms, which is very practical in applications. The formalism for parametrized $\mathfrak{g}$-valued differential forms can be easily obtained by the hint given by considering the pull-back of given expressions for particular sections $u\in \Gamma(\pi)$. 

 Subsequently, the dependence of all structures on the corresponding variables will be omitted to simplify the notation. We recall that $G$ and $\mathfrak{g}$ are assumed to be finite-dimensional matrix Lie groups and Lie algebras respectively. 
 
Consider a deformation
$$
\omega_\epsilon:=\omega+\epsilon \Upsilon,\qquad  \Phi_\epsilon:=R_{\exp(\epsilon F)}\Phi, \qquad \forall\epsilon\in\mathbb{R},\forall \lambda\in\Lambda,
$$
where $\omega,\Upsilon\in \Omega^1_\lambda(\mathcal{E}^\infty)\otimes \mathfrak{g}$, $\Phi\in C_\lambda^\infty(\mathcal{E}^\infty)\otimes G $ and $F\in C_\lambda^\infty(\mathcal{E}^\infty)\otimes\mathfrak{g}$ satisfy the conditions 
\begin{equation}\label{GenCon}
{\rm \bf d}_{\omega_\epsilon}\omega_\epsilon=0,\qquad R_{{\Phi}^{-1}_{\epsilon}}{\bf d}{\Phi}_\epsilon=\omega_\epsilon, \qquad \forall\epsilon\in\mathbb{R},\forall \lambda\in\Lambda.
\end{equation}
Although the above spectral differential forms may admit extensions to $\Omega_\lambda(J^\infty)\otimes\mathfrak{g}$, e.g. the above $\omega$ is given, in applications, by the restriction to $\mathcal{E}^\infty$ of the spectral differential form related to a LSP, we will be mainly concerned with their values on-shell. 
The first condition in (\ref{GenCon}) implies that
$$
{\rm \bf d}_{ \omega_\epsilon}
\omega_\epsilon={\rm \bf d}\omega_\epsilon-\frac 12[\omega_\epsilon\wedge\omega_\epsilon]=0,\qquad \forall \epsilon \in \mathbb{R},\forall \lambda\in \Lambda.
$$
The latter is satisfied if and only if  ${\rm \bf d}_{\omega}\omega=0$ on-shell and one has that
$$
0\!=\!\frac{\rm d}{{\rm d}\epsilon }\bigg|_{\epsilon=0}\!\!\!\left({\rm \bf d}\omega_\epsilon\!-\!\frac 12[\omega_\epsilon\!\wedge\!\omega_\epsilon]\right)\!=\!{\rm\bf d}\Upsilon-[\omega\wedge \Upsilon]\!=\!{\rm \bf d}_{2\omega}\Upsilon,
$$
where 
$\Upsilon:=A_\alpha{\rm d}x^\alpha
$ in coordinates. Under previous conditions, it can be proved that ${\bf d}F={\rm Ad}_{\Phi^{-1}}\Upsilon$ for a certain $F\in C_\lambda^\infty(\mathcal{E}^\infty)\otimes \mathfrak{g}$ (cf. \cite{Grundland15,GLM16}). The main task now is to provide methods to obtain $F$ from a certain ${\bf d}_{2\omega}$-closed $\Upsilon$.

Since our deformation does not transform $\lambda$, the singularity structure of the $\omega_\epsilon$ in the spectral parameter $\lambda$ remains untouched. 
Each $\omega$ and $\Upsilon$ satisfying the above equations generate a submanifold immersed in the Lie algebra $\mathfrak{g}$ as explained next in this section. It is worth noting that our geometric approach allows us to simplify previous proofs in the literature (see for instance \cite{FG96,Grundland15,GLM16}). Moreover, our proof is general for any arbitrary $M$ and deals not only with $M=\mathbb{R}^2$ as is standard in the literature \cite{Ci97,Ci07,DS92,FG96,Grundland15,GLM16}.

We then say that $\Phi$ is an {\it on-shell spectral solution of the LSP problem} when $\Phi([u],\lambda)$ is a particular solution of (\ref{LSPn}) for every $\lambda$ and every $[u]\in\mathcal{E}^\infty$. We may assume that $\Phi$ is extended to other $[u]$ out of $\mathcal{E}^\infty$ but in that case we cannot ensure that $\Phi([u],\lambda)$ will be a solution of (\ref{LSPn}) for all values of $\lambda\in \Lambda$. In such a case, we say that $\Phi([u],\lambda)$ is an {\it off-shell} spectral solution of (\ref{PDE}).

\begin{lemma}\label{lem:closed} Let $\Phi$ be an on-shell solution of the jet LSP (\ref{LSPn}) and let $\omega$ be the associated element in $\Omega_\lambda^1(J^\infty)\otimes\mathfrak{g}$ satisfying ${\bf d}_{\omega}\omega=0$ for every $\lambda\in \Lambda$ on-shell. Then, 
\begin{equation}\label{Exp}
{\rm \bf d}{\rm Ad}_{\Phi^{-1}}\vartheta={\rm Ad}_{\Phi^{-1}}{\rm \bf d}_{2\omega}\vartheta,\qquad \forall \vartheta\in \Omega_\lambda(\mathcal{E}^\infty)\otimes \mathfrak{g}.
\end{equation}
\end{lemma}
\begin{proof} Let us prove (\ref{Exp}) for a spectral differential k-form, say $\vartheta:=f^j_J{\rm d}x^J\otimes v_j\in \Omega^k_\lambda(\mathcal{E}^\infty)\otimes \mathfrak{g}$, where $v_1,\ldots, v_r$ stands for a basis of $\mathfrak{g}$. The validity of equality (\ref{Exp}) on the whole $\Omega_\lambda(\mathcal{E}^\infty)\otimes \mathfrak{g}$ follows from this by linearity. Now,
$$
{\rm \bf d}{\rm Ad}_{\Phi^{-1}}\vartheta=(D_\alpha f^j_J){\rm  d}x^\alpha\wedge {\rm  d}x^J\otimes {\rm Ad}_{\Phi^{-1}}(v_j)+f^j_J{\rm d}x^\alpha\wedge {\rm  d}x^J\otimes D_\alpha({\rm Ad}_{\Phi^{-1}}(v_j)).
$$
Since $D_\alpha\Phi^{-1}{\rm d}x^\alpha=-L_{\Phi^{-1}}\omega$ on-shell and writing $\omega={\rm d}x^\alpha\otimes U_\alpha $, it follows that $D_\alpha{\rm Ad}_{\Phi^{-1}}(v_j)=-{\rm Ad}_{\Phi^{-1}}([U_\alpha,v_j])$ and 
$$
{\rm \bf d}{\rm Ad}_{\Phi^{-1}}\vartheta={\rm Ad}_{\Phi^{-1}}({\rm \bf d}\vartheta-[\omega\wedge\vartheta])={\rm Ad}_{\Phi^{-1}}{\rm \bf d}_{2\omega}\vartheta.
$$
\end{proof}

\begin{lemma}\label{Sim}  Let $\Phi$ be an on-shell spectral solution of the jet LSP (\ref{LSPn}) and let $\omega$ be its associated spectral differential one-form in $\Omega^1_\lambda(J^\infty)\otimes\mathfrak{g}$. If $\Upsilon$ is a spectral differential one-form on-shell, then there exists a locally defined $F\in C^\infty_\lambda(\mathcal{E}^\infty)\otimes \mathfrak{g}$ such that 
\begin{equation}\label{F}
{\rm \bf d}F={\rm Ad}_{\Phi^{-1}}\Upsilon
\end{equation}
if and only if ${\rm \bf d}_{2\omega}\Upsilon=0$ on-shell.
\end{lemma}
\begin{proof} From Lemma \ref{lem:closed} and the present assumptions, it follows that 
$${\rm \bf d}{\rm Ad}_{\Phi^{-1}}\Upsilon={\rm Ad}_{\Phi^{-1}}{\rm \bf d}_{2\omega}\Upsilon
$$
on-shell. Hence, ${\rm Ad}_{\Phi^{-1}}\Upsilon$ is ${\rm\bf d}$-closed (on-shell)  if and only if ${\rm \bf d}_{2\omega}\Upsilon=0$ on-shell. From the properties of the cohomology ${\rm \bf d}$ on $\Omega_\lambda(\mathcal{E}^\infty)\otimes \mathfrak{g}$, the spectral differential 1-form on-shell ${\rm Ad}_{\Phi^{-1}}\Upsilon$ is closed if and only if there exists locally an $F\in C^\infty_\lambda(\mathcal{E}^\infty)\otimes \mathfrak{g}$ such that (\ref{F}) is satisfied. This finishes the proof.
\end{proof}

In order to fix accurately the notation of our paper, let us introduce the following notion. 

\begin{definition} We call $F$ in (\ref{F}) an {\it immersion formula} for the jet LSP (\ref{LSPn}).
\end{definition}

Lemma \ref{Sim} shows the interest of finding the on-shell spectral differential one-forms $\Upsilon$ satisfying the equation ${\rm \bf d}_{2\omega}\Upsilon=0$ on-shell for a certain $\omega\in \Omega_\lambda^1(J^\infty)\otimes \mathfrak{g}$ satisfying the condition ${\rm \bf d}_{\omega}\omega=0$ on-shell. There are certain known methods to obtain $\Upsilon$. The work \cite{KL99} employs an exact differential one-form, the {\it generalized Weierstrass embedding formula}, to obtain an immersion formula. The articles \cite{Ci97,Sy82} show that the admissible symmetries of the ZCC (\ref{sur1}) when $m=2$ include a conformal transformation of the spectral parameter $\lambda$, a gauge transformation of the wavefunction $\Phi$ in the LSP (\ref{LSP1}) and generalized symmetries of the integrable system (\ref{2.1}). All these symmetries can be used to determine explicitly an immersion formula $F$. Most works deal with applications where the image of $F$ is a two-dimensional surface. Less commonly applications deal with an $F$ whose image is an immersed submanifold in a Lie algebra \cite{Br13}. All such results can be found in a simpler and more general way, e.g. allowing the image of $F$ to be a general immersed submanifold in $\mathfrak{g}$, by using the following theorem.

\begin{theorem}\label{MT}
Let $\Phi$ be an on-shell spectral solution  of the associated LSP (\ref{LSPn}) and let $\omega\in \Omega_\lambda^1(J^\infty)\otimes\mathfrak{g}$ be the associated spectral differential form. Assume that $\Upsilon \in \Omega_\lambda^1(\mathcal{E}^\infty)\otimes \mathfrak{g}$ takes the form on-shell
\begin{equation}
\Upsilon:=\beta(\lambda)\partial_\lambda \omega+{\rm \bf d}_{2\omega}S
+\mathcal{L}_{{\rm pr}\,X_R}\omega.\label{2.12}
\end{equation}
Here, $\beta(\lambda)$ is an arbitrary scalar function depending only on $\lambda$, the function $S$ is an arbitrary element of $C_\lambda^\infty(\mathcal{E}^\infty)\otimes\mathfrak{g}$, and the vector field $X_R:=R^i\partial_{u^i}$ is a generalized symmetry of  $\Delta[j^p_xu]=0$. Then, there exists an immersion formula $F\in C^\infty_\lambda(\mathcal{E}^\infty)\otimes \mathfrak{g}$ given by the formula
\begin{equation}
F=\beta(\lambda)L_{\Phi^{-1}} \partial_\lambda\Phi+{\rm Ad}_{\Phi^{-1}}S+R_{\Phi^{-1}}{\rm pr}\,X_R\Phi.\label{2.14}
\end{equation}
\end{theorem}
\begin{proof} Due to historical reasons, let us prove that the three terms 
$$
\Upsilon^{ST}:=\beta(\lambda)\partial_\lambda\omega,\qquad \Upsilon^{CD}:={\rm \bf d}_{2\omega}S,\qquad \Upsilon^{FG}:=\mathcal{L}_{{\rm pr}\,X_R}\omega
$$
of the spectral $\mathfrak{g}$-valued differential one-form on-shell (\ref{2.12}) are such that the action of ${\rm Ad}_{\Phi^{-1}}$ on each one is the differential (relative to ${\bf d}$) of the $\mathfrak{g}$-valued spectral functions on-shell
$$
F^{ST}:=\beta(\lambda)L_{\Phi^{-1}} \partial_\lambda\Phi,\quad F^{CD}:={\rm Ad}_{\Phi^{-1}}S,\quad F^{FG}:=L_{\Phi^{-1}}{\rm pr}\,X_R \Phi,
$$
respectively. This proves that the latter $\mathfrak{g}$-valued spectral functions are immersion formulas for (\ref{2.12}).

Let us start by studying the {\it Sym-Tafel spectral differential 1-form}, namely $\Upsilon^{ST}$. Using the fact that $\Phi$ is an on-shell spectral solution of (\ref{LSP1}), we obtain on-shell
$$
\begin{aligned}
{\bf d} F^{ST}&=D_\alpha (F^{ST}){\rm d}x^\alpha \\
&=\beta(\lambda)[(D_\alpha \Phi^{-1})\partial_\lambda \Phi+\Phi^{-1}D_\alpha \partial_\lambda \Phi]{\rm d}x^\alpha\\&=\beta(\lambda)[-\Phi^{-1}U_\alpha\partial_\lambda \Phi+\Phi^{-1}\partial_\lambda (U_\alpha \Phi)]{\rm d}x^\alpha\\&=\beta(\lambda){\rm Ad}_{\Phi^{-1}}\partial_\lambda \omega\\
&={\rm Ad}_{\Phi^{-1}}\Upsilon^{ST}.
\end{aligned}
$$

Let us now study the {\it Cie\'sli\'nski-Doliwa 1-form}, i.e. $\Upsilon^{CD}={\rm \bf d}_{2\omega}S$. From Lemma \ref{lem:closed}, it follows that
$$
{\bf d}F^{CD}={\bf d}{\rm Ad}_{\Phi^{-1}}S={\rm Ad}_{\Phi^{-1}}{\bf d}_{2\omega}S={\rm Ad}_{\Phi^{-1}}\Upsilon^{CD}.
$$

The {\it Fokas-Gel'fand 1-form}, i.e. $\Upsilon^{FG}=\mathcal{L}_{{\rm pr}X_R}\omega$, induces a function $F^{FG}\in C^\infty_\lambda(\mathcal{E}^\infty)\otimes\mathfrak{g}$. To show this it is enough to observe that
since $[D_\alpha,{\rm pr}X_R]=0$ one has on-shell that
$$
\begin{aligned}
{\bf d}F^{FG}&=(-\Phi^{-1}U_\alpha{\rm pr}X_R\Phi+L_{\Phi^{-1}}{\rm pr}X_R(U_\alpha \Phi)){\rm d}x^\alpha\\
&={\rm Ad}_{\Phi^{-1}}\Upsilon^{FG}.
\end{aligned}
$$

\end{proof}

\section{Relations between immersion formulas.}\setcounter{equation}{0}
The work \cite{GLM16} showed that there exists a gauge transformation mapping the Sym--Taffel and Cie\'sli\'nski--Doliwa immersion formulas to each other when the potential, $S$, of one of them is invertible. This result is an immediate consequence of our cohomology approach, as shown next. 

\begin{proposition}\label{eq1}
A vector field $D_\lambda$ is a symmetry of the ZCC condition ${\rm \bf d}_{\omega}\omega=0$ associated with $\Delta(j_x^pu)=0$  if and only if there exists an $S^{ST}\in C^\infty_\lambda(\mathcal{E}^\infty)\otimes\mathfrak{g}$ such that on-shell
\begin{equation}\label{3.1}
{\rm \bf d}_{2\omega} S^{ST}=\beta(\lambda)\mathcal{L}_{D_\lambda}\omega.
\end{equation}
\end{proposition}
\begin{proof}
Lemma \ref{Sim} establishes that ${\rm Ad}_{\Phi^{-1}}\Upsilon={\rm \bf d}F$ locally for a certain $F\in C_\lambda^\infty(\mathcal{E}^\infty)\otimes \mathfrak{g}$ if and only if ${\bf d}_{2\omega}\Upsilon=0$ on-shell. Since $[{\bf d},\mathcal{L}_{D_\lambda}]=0$ off-shell, one has
\begin{multline}\label{eq:lambda}
{\rm \bf d}_{2\omega}\left[\beta(\lambda)\mathcal{L}_{D_\lambda}\omega\right]=\beta(\lambda)\left[{\rm \bf d}\mathcal{L}_{D_\lambda}\omega-\left[\omega\wedge \mathcal{L}_{D_\lambda}\omega\right]\right]=\\\beta(\lambda)\left[\mathcal{L}_{D_\lambda}{\rm \bf d}\omega-\frac 12 \mathcal{L}_{D_\lambda}[\omega\wedge\omega]\right]=\beta(\lambda)\mathcal{L}_{D_\lambda}{\rm \bf d}_\omega\omega.
\end{multline}
It is worth noting that although ${\bf d}_\omega\omega=0$ on-shell, $D_\lambda{\bf d}_\omega\omega$ may not vanish, as these terms give the infinitesimal variation of ${\bf d}_\omega\omega$ along a one-parametric family of curves $(u_\epsilon(t),\lambda_\epsilon)$ generated by the flow of $D_\lambda$. Only $u_0(t)$ must be a solution to $\Delta(j^p_xu)=0$, while  ${\bf d}_\omega\omega$ may not vanish on $(u_\epsilon(x),\lambda_\epsilon)$.

If  $\beta(\lambda)\mathcal{L}_{D_\lambda}\omega={\bf d}S^{ST}$ on-shell, then (\ref{eq:lambda}) vanishes and $\mathcal{L}_{D_\lambda}{\bf d}_\omega\omega=0$ on-shell, i.e. ${D_\lambda}$ is symmetry of the ZCC condition. Conversely, if $D_\lambda$ is a symmetry of ZCC condition for the $\Delta(j_x^pu)$, then the expression (\ref{eq:lambda}) vanishes on-shell and, due to the properties of the ${\bf d}_{2\omega}$ cohomology, there exists a $S^{ST}\in C^\infty_\lambda(\mathcal{E}^\infty)\otimes\mathfrak{g}$ such that ${\bf d}S^{ST}=\beta(\lambda)\mathcal{L}_{D_\lambda}\omega$ on-shell.
\end{proof}

It is worth noting that the cohomological proof of this theorem is almost trivial. Meanwhile, its counterpart in coordinates in \cite{GLM16} is much longer and obscure.

Analogously, it is possible to see that the remaining terms within the $\mathfrak{g}$-valued spectral differential one-form $\Upsilon$ in Theorem \ref{MT} are induced from other types of gauges. Then, these gauges allow us to map one term into another.

\begin{proposition}\label{Eq2}
A vector field $X_R$ is a generalized symmetry of the ZCC ${\bf d}_\omega\omega$ of the jet LSP associated with an integrable system $\Delta[j_x^pu]=0$ if and only if there exists a spectral function $S^{FG}\in C_\lambda^\infty(\mathcal{E}^\infty)\otimes\mathfrak{g}$ such that on-shell
\begin{equation}
{\bf d}_{2\omega} S^{FG}=\mathcal{L}_{{\rm pr}\,X_R}\omega.\label{4.1}
\end{equation}
\end{proposition}

\begin{proof} 
Using the Lemma \ref{lem:closed}, we obtain that
$
{\bf d}_{2\omega}\Upsilon^{FG}={\rm Ad}_\Phi{\bf d}{\rm Ad}_{\Phi^{-1}}\Upsilon^{FG}.
$ A short computation shows that

$$
\begin{aligned}
{\bf d}{\rm Ad}_{\Phi^{-1}}\Upsilon^{FG}&=D_\beta[(\Phi^{-1}(\mathcal{L}_{{\rm pr}_{X_R}}D_\alpha\Phi)-\Phi^{-1}U_\alpha(\mathcal{L}_{{\rm pr}_{X_R}}\Phi)]{\rm d}x^\beta\wedge {\rm d}x^\alpha\\
&=(\Phi^{-1}\mathcal{L}_{{\rm pr}_{X_R}}D_\beta D_\alpha\Phi+\Phi^{-1}U_\beta U_\alpha\mathcal{L}_{{\rm pr}_{X_R}}\Phi-\Phi^{-1}D_\beta U_\alpha \mathcal{L}_{{\rm pr}_{X_R}}\Phi){\rm d}x^\beta\wedge{\rm d}x^\alpha\\
&=\sum_{\alpha<\beta}(\Phi^{-1}\mathcal{L}_{{\rm pr}_{X_R}}(D_\beta D_\alpha-D_\alpha D_\beta)\Phi-\Phi^{-1}({\bf d}_\omega\omega)\mathcal{L}_{{\rm pr}_{X_R}}\Phi){\rm d}x^\beta\wedge{\rm d}x^\alpha\\
&=\Phi^{-1}\mathcal{L}_{{\rm pr}_{X_R}}({\bf d}_\omega\omega).
\end{aligned}
$$

If $\Upsilon^{FG}={\bf d}_{2\omega}S^{FG}$ on-shell, then $\mathcal{L}_{{\rm pr}_{X_R}}{\bf d}_\omega\omega=0$ on-shell and the generalized symmetry ${X_R}$ is a symmetry of $\Delta(j^p_xu)$. Conversely, if ${X_R}$ is a generalized symmetry of $\Delta(j_x^pu)$, then  ${\bf d}_{2\omega}\Upsilon^{FG}=0$ and $\Upsilon^{FG}={\bf d}_{2\omega}S^{FG}$ on-shell for a certain $S^{FG}\in C^\infty_\lambda(\mathcal{E}^\infty)\otimes\mathfrak{g}$.

\end{proof}

Let us prove a last result allowing us to map different immersion maps among themselves. 

\begin{corollary}\label{Eq3} Let $\Upsilon_1,\Upsilon_2\in \Omega_\lambda^1(\mathcal{E}^\infty)\otimes \mathfrak{g}$ be two solutions of ${\rm \bf d}_{2\omega}\Upsilon=0$ on-shell. Then, there exists an element $F\in C^\infty_\lambda(\mathcal{E}^\infty)\otimes {\rm Aut}(\mathfrak{g})$, a so-called {\it gauge transformation}, mapping the immersion formula $S_1$ of $\Upsilon_1$ onto the immersion formula $S_2$ of $\Upsilon_2$ or vice versa, whenever $S_1$ or $S_2$ are invertible matrices of $\mathfrak{g}$, respectively.
\end{corollary}
\begin{proof} It is immediate that under the given assumptions, $F=S_2\circ S_1^{-1}$ or $F=S_1\circ S_1^{-2}$ is the required element for invertible $S_1$ or $S_2$, respectively. 
\end{proof}

In particular, the application of the above corollary for the Sym-Tafel and Fokas-Gelf'and terms is summarised in the following diagram.

\begin{figure}[h!]
\setlength{\unitlength}{0.14in}
\centering
\begin{picture}(30,12)
\put(1,6){$\Phi\in G$}
\put(1.75,5.5){\vector(3,-1){10.5}}
\put(1.75,7.3){\vector(3,1){9.5}}
\put(4,9.5){$S_1\in\mathfrak{g}$}
\put(4,3){$S_2\in\mathfrak{g}$}
\put(9.5,11){$F^{ST}=\beta(\lambda)\Phi^{-1}(D_\lambda\Phi)\in\mathfrak{g}$}
\put(10.25,1){$F^{FG}=\Phi^{-1}(\mbox{pr}\omega_R\Phi)\in\mathfrak{g}$}
\put(10,6){$S_1\circ S_2^{-1}$}
\put(15,2){\vector(0,1){8.5}}
\put(16,10.5){\vector(0,-1){8.5}}
\put(17,6){$S_2\circ S_1^{-1}$}
\end{picture}
\caption{Representation of the relations between the wavefunction $\Phi\in G$ and the $\mathfrak{g}$-valued ST and FG immersion formulas for soliton submanifolds.}
\label{Fig1}
\end{figure}

\section{Immersion formulas for PDEs with initial conditions}
It frequently happens in applications that an LSP only describes a subspace of particular solutions of a system of PDEs. For instance, it may happen that the the ZCC of the LSP just recovers the particular solutions of the system of PDEs satisfying a particular initial condition. The aim of this section is to recover immersion formulas and other results in this case.  To illustrate the problem under consideration, let us analyse the partial differential equation in $J^2$ for $\pi:(u;x,y)\in\mathbb{R}\times\mathbb{R}^2\mapsto (x,y)\in \mathbb{R}^2$ given by
\begin{equation}\label{Exp2}
u_{xx}-\frac 12f'(u)=0,
\end{equation}
for some function $f:u\in \mathbb{R}\mapsto f(u)\in \mathbb{R}$. It is straightforward to prove that the function $
u_x^2-f(u)$
is a constant of motion of (\ref{Exp2}). Nevertheless, not every solution of our PDE satisfies $u_x^2=f(u)$, which can be understood as a boundary condition. For instance,  if $f(u)=0$ the function $u(x,y)=x$ satisfies (\ref{Exp2}) without obeying the given initial condition. The PDE (\ref{Exp2}), along with the previous initial condition, admits a Lax pair 
$$
\partial_x\Phi=L\Phi,\qquad \partial_y\Phi=M\Phi,\qquad \Phi\in SL(2,\mathbb{R}), L,M\in \mathfrak{sl}(2,\mathbb{R}),
$$
where
$$
L:=\frac 12\left[\begin{array}{cc}0&\frac{f'(u)}{u+\lambda}-\frac{f(u)-g(-\lambda)}{(u+\lambda)^2}\\1&0\end{array}\right],\qquad M:=\frac 12\left[\begin{array}{cc}u_x&-\frac{f(u)-g(-\lambda)}{u+\lambda}\\{u+\lambda}&-u_x\end{array}\right].
$$

The  system (\ref{Exp2}) admits the generalized Lie symmetries 
$$
Q_1:=u_x\frac{\partial}{\partial u}\qquad Q_2:=u_x\left(\int f^{-3/2}(u)dx\right)\frac{\partial}{\partial u}.
$$
Since ${\rm pr} \,\,Q_1-D_x=\partial_x$, the vector field $Q_1$ induces a translation in the $x$ variable, which is a symmetry of the initial condition $u_x^2-f(u)$ and (\ref{Exp2}). Therefore, we can redefine our PDE by adding the initial condition. This enables us to apply all our previous results to this new problem. 

On the other hand, a straightforward calculation shows that $Q_2$ is an infinitesimal symmetry only of (\ref{Exp2}), i.e. it does not leave invariant the initial condition. As a consequence, the formalism given in the previous section does not apply to this new case. Despite this, it is possible to rewrite all previous results to deal with it.

Let us consider a system of PDEs given by (\ref{PDE}) and a LSP which only applies to solutions of (\ref{PDE}) with some additional particular conditions. Let the new system by described by $\mathcal{S}'\subset J^p$. The parametrized $\mathfrak{g}$-valued differential forms of $\Omega_{\mathcal{S}}(M)\otimes \mathfrak{g}$ can be reduced to the set $\mathcal{S}'$. The restrictions of Theorem \ref{Poincare1} and Corollaries \ref{DoPragi} and \ref{DoPragi2} to $\mathcal{S}'$ remain true. This allows us to restrict our cohomologies, ${\bf d}$ and ${\bf d}_{2\omega}$ to  $\Omega_{\mathcal{S}'}(M)\otimes \mathfrak{g}$. Moreover, Lemmas \ref{lem:closed} and \ref{Sim} remain true when $\Phi$ is a solution of the LSP for curves $u(x)$ in $\mathcal{S}'$.

Let us now comment on spectral differential forms. If we consider the initial condition as a part of the initial system of PDEs, the space of solutions of the new system of PDEs with the initial conditions induces a submanifold $\mathcal{E}_2\subset \mathcal{E} \subset J^p$ and a new prolongation $\mathcal{E}_2^\infty$. Hence, all previous results for $\mathcal{E}$ now apply to $\mathcal{E}_2$. Since they can be trivially restricted to the systems of PDEs on $\mathcal{E}_2$, the restriction of the Cieli\'nski--Doliwa or the Sym--Tafel immersion formulas remain valid for the restriction of the initial system. 

In spite of previous arguments, there exists one important difference relative to previous immersion formulas in the case of the restriction. The symmetries of $\mathcal{E}$ do not necessarily give rise to symmetries of the associated LSP. Consequently, the FG immersion formula must be modified accordingly. This will be solved by adding a new term in the Fokas--Gelf'and one-form. This is explained in the next theorem.

\begin{theorem}\label{ModFG}
Let $\Phi$ be an on-shell spectral solution  of the associated LSP (\ref{LSP1}) defined over a submanifold of solutions of the initial systems of PDEs and let $\omega\in \Omega_\lambda^1(J^\infty)\otimes\mathfrak{g}$ be the associated $\mathfrak{g}$-valued spectral differential form. Assume that $\Upsilon^{FG'} \in \Omega^1_\lambda(J^\infty)\otimes \mathfrak{g}$ takes the form on-shell
\begin{equation}
\Upsilon^{FG'}:=\mathcal{L}_{{\rm pr}\,X_R}\omega+R_{\Phi^{-1}}{{\rm pr}\,X_R}(D_\alpha\Phi-U_\alpha \Phi).\label{2.123}
\end{equation}
Here, the vector field $X_R$ is a generalized symmetry of  $\Delta(j^p_xu)=0$. Then, there exists an immersion formula for $\Upsilon^{FG'}$ given by 
\begin{equation}
F^{FG'}=L_{\Phi^{-1}}{\rm pr}\,X_R\Phi\in C^\infty_\lambda(\mathcal{E}_2^\infty)\otimes \mathfrak{g}\label{2.143}
\end{equation}
\end{theorem}
\begin{proof} As in Theorem \ref{MT}, let us prove that 
$$
\Upsilon^{FG'}:=\mathcal{L}_{{\rm pr}\,X_R}\omega+R_{\Phi^{-1}}{{\rm pr}\,X_R}(D_\alpha\Phi-U_\alpha \Phi){\rm d}x^\alpha,
$$
is such that the action of ${\rm Ad}_{\Phi^{-1}}$ on it gives the differential (relative to ${\bf d}$) of the $\mathfrak{g}$-valued spectral function (\ref{2.143}). To show this it is enough to recall that
since $[D_\alpha,{\rm pr}X_R]=0$ and then
$$
\begin{aligned}
{\bf d}F^{FG'}&=(-L_{\Phi^{-1}}U_\alpha{\rm pr}X_R\Phi+L_{\Phi^{-1}}{\rm pr}X_R(D_\alpha \Phi-U_\alpha \Phi)+L_{\Phi^{-1}}{\rm pr}X_R(U_\alpha \Phi)){\rm d}x^\alpha\\
&={\rm Ad}_{\Phi^{-1}}(\mathcal{L}_{{\rm pr}X_R}\omega+R_{\Phi^{-1}}{\rm pr}X_R(D_\alpha \Phi-U_\alpha \Phi))\\
&={\rm Ad}_{\Phi^{-1}}\Upsilon^{FG'}.
\end{aligned}
$$
\end{proof}

\section{Soliton surfaces for $\mathbb{C}P^{N-1}$ sigma models}
Let us now illustrate our techniques by studying immersion formulas for the general $\mathbb{C}P^{N-1}$ sigma model. Based on the Gram-Schmidt orthogonalization procedure, a method for constructing an entire class of solutions on $S^2$ admitting the finite action of the $\mathbb{C}P^{N-1}$ sigma model was proposed by A. Dim, Z Horvarth and D. Zakrzewski \cite{DHZ84,DZ80} and latter studied by R. Sasaki \cite{Sa83}. In particular, we focus on the case $N=3$, whose immersion formulas will allow us to study the so-called {\it Veronese surfaces} and the mixed type soliton surfaces \cite{BBT06,BJRW88,Zakrzewski89}. A final type of surface induced by the aforesaid will be also briefly commented on (for details on these models see \cite{BBT06,He01,MS04,ZM79,Zakrzewski89}).
 
The most fruitful technique for the study of integrable surfaces for $\mathbb{C}P^{N-1}$ models has been formulated through descriptions of the model via rank-one Hermitian projectors. A matrix $P(z,\bar z)$, with $z\in \mathcal{R}:=\mathbb{C}\cup \{\infty\}$, is said to be a {\it rank-one Hermitian projector} if
\begin{equation*}
P^2=P,\qquad  P=P^\dagger,\qquad \mbox{tr}\,P=1.
\end{equation*}
The image of the projector $P$ is determined by a complex one-dimensional space within $\mathbb{C}^N$. Therefore, there exists a $\mathbb{C}^N$-valued function $$f:\mathcal{R}\supseteq\Omega\ni \xi:=x+iy\mapsto (f_0(\xi),f_1(\xi),...,f_{N-1}(\xi))^T\in\mathbb{C}^N\backslash\lbrace0\rbrace$$ such that
\begin{equation}
P=\frac{f\otimes f^\dagger}{f^\dagger f}.\label{Pf}
\end{equation}
The equations of motion for the $\mathbb{C}P^{N-1}$ model read
\begin{equation}\label{l}
\Omega(P):=[(\partial_x^2+\partial_y^2)P,P]=0.
\end{equation}
For our purposes of investigating infinitesimal symmetries and immersion formulas related to this model,
a more appropriate form of the above expression is given by the differential equation on the second-order jet $J^2$ of the bundle $\pi:(\theta,\xi)\in \mathfrak{su}(N)\times \mathcal{R}\mapsto \xi\in \mathcal{R}$ given by \begin{equation}
\left[(\partial_x^2+\partial_y^2)\theta,\theta\right]=0\label{EL},
\end{equation}
where $\theta:={\rm i}(P-{\rm Id}_N/N)\in \mathfrak{su}(N)$. This differential equation induces in the standard way a submanifold of $\mathcal{E}\subset J^2$ \cite{Ol93}. To apply our approach, it is necessary to extend this differential equation to $J^\infty$ by adding, as standard, the conditions $D_J\left[(\partial_1^2+\partial_2^2)\theta,\theta\right]=0$, for an arbitrary multi-index $J$ \cite{Vi01}. This gives rise to a differential equation on $J^\infty$ given by
\begin{equation}\label{Jinft}
[(D^2_x+D_y^2)\theta,\theta]=0,\qquad D_J\left[(\partial_x^2+\partial_y^2)\theta,\theta\right]=0,\qquad \forall J,
\end{equation}
which in turn gives rise to a submanifold $\mathcal{E}^\infty$ of $J^\infty$ \cite{Vi01}.

The LSP associated with (\ref{Jinft}) is given by \cite{Mi86,ZM79}
\begin{equation}\label{2.6}
\begin{gathered}
D_\alpha\Phi=\hat U_{\alpha }\Phi,\qquad \alpha\in \{x,y\},\qquad\lambda=it,\qquad t\in \mathbb{R},\\
\hat U_1:=\frac{-(1+i\lambda)}{1-\lambda^2}[(D_x+D_y)\theta,\theta],\quad
\hat U_2:=\frac{1+i\lambda}{1-\lambda^2}[(D_x-D_y)\theta,\theta],
\end{gathered}
\end{equation}
with a soliton solution on-shell $\Phi=\Phi([\theta],\lambda)\in SU(N)$ which goes to ${\rm Id}_N$ as $\lambda\rightarrow\infty$ and with the spectral functions $\hat U_x([\theta],\lambda),\hat U_y([\theta,\lambda])$ taking values in $ \mathfrak{su}(N)$ \cite{ZM79,Zakrzewski89}. This system is related to the  $\mathfrak{su}(N)$-valued spectral differential one-form
$$
\omega:={\rm d}x\otimes \hat U_x+{\rm d}y\otimes \hat U_y.
$$
In fact, the integrability condition for (\ref{2.6}) reads as
$$
{\bf d}_\omega\omega={\rm d}x\wedge {\rm d}y[(D_x^2+D_y^2)\theta,\theta]).
$$
One sees that the above vanishes exactly on the space of solutions $\theta(\xi)$ corresponding to (\ref{Jinft}). Additionally, the above spectral $\mathfrak{su}(N)$- differential form vanishes on-shell, i.e. on the submanifold $\mathcal{E}^\infty\subset J^\infty$.

To study the immersion submanifolds for (\ref{EL}), it is interesting to obtain a spectral solution of (\ref{2.6}). This can be done by adapting to $J^\infty$ the rising and lowering projectors for $\mathbb{C}P^{N-1}$ models \cite[p. 187]{GG10}, namely
\begin{equation}
\Pi_-(\theta):=\frac{\bar{D}\theta(\mathcal{F}-i\theta)D\theta}{\tr(\bar{D}\theta(\mathcal{F}-i\theta)D\theta)},\qquad
\Pi_+(\theta):=\frac{D\theta(\mathcal{F}-i\theta)\bar{D}\theta}{\tr(D\theta(\mathcal{F}-i\theta)\bar{D}\theta)},
\qquad \mathcal{F}:=\frac{1}{N}{\rm I}_N,
\end{equation}
where the traces in the denominators are different from zero for $\theta\neq 0$ unless the whole matrix is zero and $D:=(D_x+iD_y)/2$ and $\bar D:=(D_x-iD_y)/2$. Then, a family of spectral solutions $\Phi_k([\theta],\lambda)$ to the LSP (\ref{2.6}) read as (cf. \cite{GLM16})
\begin{equation}\label{parsolinf}
\begin{array}{l}
\Phi_k([\theta],\lambda):={\rm Id}_N+\frac{4\lambda}{(1-\lambda)^2}\sum_{j=0}^{k-1}\Pi_-^j(\theta)-\frac{2}{1-\lambda}\left(\frac{1}{N}{\rm Id}_N-i\theta\right)\in SU(N),
\end{array}
\end{equation}
where $k=\overline{0,N-1}$. The function $\Phi_0([\theta],\lambda)$ is holomorphic, $\Phi_{N-1}([\theta],\lambda)$ is anti-holomorphic and the remaining ones $k=\overline{1,N-2}$ are mixed.

For any real functions $f$ and $g$ of one variable, the equations of motion (\ref{EL}) admit the Lie symmetries given by the prolongations to $J^2$ of the vector fields
\begin{equation}
X_{C}=\left[f(x)\partial_x\theta^j+g(y)\partial_y\theta^j\right]\frac{\partial}{\partial\theta^j},
\end{equation}
which are called {\it conformal} as they vanish on functions depending only on the independent variables $x,y$. 
The integrated form of the immersed surface in $\mathfrak{su}(N)$ induced by the above symmetry is given by the FG formula \cite[p. 15]{GPR14}. Using (\ref{parsolinf}) for $k=0$, we obtain
\begin{equation}
F^{FG}=\Phi^{-1}\left(f(x)\hat U_x+g(y)\hat U_y\right)\Phi\in\mathfrak{su}(N).
\end{equation}

There is another possible approach to take into account. It turns out that the spectral $\mathfrak{su}(N)$-valued differential form
$$
\Upsilon:={\rm Ad}_{\Phi^{-1}}(-[\partial_x\theta,\theta]{\rm d}x+[\partial_y\theta,\theta]{\rm d}y)
$$
satisfies ${\bf d}_{2\omega}\Upsilon=0$ on-shell. As a consequence, it provides a deformation giving rise to an immersion surface of the form ${\rm d}F=-[\partial_x\theta,\theta]{\rm d}x+[\partial_y\theta,\theta]{\rm d}y$. Hence, $F$ is given by integrating this form
$$
F(x,y)=\int_\gamma(-[\partial_x\theta,\theta]{\rm d}x+[\partial_y\theta,\theta]{\rm d}y).
$$
This is indeed the expression in the set of coordinates $\{x,y\}$ on $\mathcal{R}$ of the {\it Weirstrass immersion formula} for $\mathbb{C}P^{N-1}$ models \cite{KL99}.

\subsection{Soliton surfaces associated with the $\mathbb{C}P^2$ sigma model}
Let us analyse the geometric properties of immersion surfaces for the particular $\mathbb{C}P^{N-1}$ sigma model with $N=3$ by using our techniques. The obtained surface is one of the elements of the so-called {\it Veronese sequence} \cite{BJRW88,DV00}. In this case, the projector $P$ given by (\ref{Pf}) for $k=0$ can be written in terms of the holomorphic function $f_0:=(1,\sqrt{2}\xi,\xi^2)$. One obtains that
$$
P_0=\frac{f_0\otimes f_0^\dagger}{f_0^\dagger f_0}=\frac{1}{(1+|\xi|^2)^2}\left(\begin{array}{ccc}1&\sqrt{2}\bar \xi&\bar \xi^2\\\sqrt{2}\xi&2|\xi|^2&\sqrt{2}|\xi|^2\bar \xi\\\xi^2&\sqrt{2}|\xi|^2\xi&|\xi|^4\end{array}\right).
$$
Hence, $X_0=i\left({\rm Id}_3/3-P_0\right)$ and therefore
$$
X_0=\frac{i}{(1+|\xi|^2)^2}\left(\begin{array}{ccc}\frac 13(|\xi|^4+2|\xi|^2-2)&-\sqrt{2}\bar \xi&-\bar \xi^2\\ -\sqrt{2}\xi&\frac 13 (|\xi|^4-4|\xi|^2+1)&-\sqrt{2}|\xi|^2\bar \xi\\ -\xi^2&-\sqrt{2}|\xi|^2\xi&\frac 13(1-2|\xi|^4+2|\xi|^2)\end{array}\right).
$$
This is an immersion surface which is an anti-Hermitian solution of the equation:
\begin{equation}\label{sol}
X_0^2+\frac{i}{3}X_0+\frac 29{\rm Id}_3=0.
\end{equation}
By considering the basis $S_i$ for $\mathfrak{su}(3)$ given by the orthonormalization of the basis in (\cite[p. 11]{GSZ06}) relative to the Killing metric on $\mathfrak{su}(3)$ of the form $\langle A,B\rangle:=-{\rm Tr}(AB)/2$, we obtain that every $X_0=\sum_{\alpha=1}^8x_\alpha S_\alpha$ satisfying the matrix equation (\ref{sol}) is such that its coordinates obey the following independent equations: 
\begin{eqnarray}
\sum_{i=1}^8x_i^2=\frac 13,\label{eq1}\\
\frac 43 \left(x_4-\frac {1}{4\sqrt{3}}\right)^2+\sum_{i=5}^8x_i^2=\frac 14,\label{eq2}\\
\sum_{i=1}^3x_i^2+\frac 13x_4^4+x_6^2+x_8^2+\frac 13x_3+\frac 1{3\sqrt{3}}(6x_3+1)x_4=\frac 29,\label{eq3}\\
-3x_3^3-6(\sqrt{3}x_4+1)x_3-9x_4^2+2\sqrt{3}x_4+1=0.\label{eq6}
\end{eqnarray}
The first one is obtained by considering the trace of (\ref{sol}). The second and third ones come from analyzing the diagonal elements of (\ref{sol}). Other  conditions on the coordinates of $X_0$, e.g.
\begin{eqnarray}
-\frac 13x_2(2\sqrt{3}x_4+1)-x_6x_7+x_5x_8=0,\label{eq4}\\ 
\frac 13(2\sqrt{2}x_4+1)x_1-x_5x_6-x_7x_8=0,\label{eq5}
\end{eqnarray}
arise from analyzing all entries of the equation (\ref{sol}).

Previous conditions allow us to determine the shape of the immersed surface. The equation (\ref{eq1}) represents a sphere in $\mathbb{R}^8$. To obtain the Veronese surface one has to intersect the sphere (\ref{eq1}) by an ellipsoidal cylinder described by (\ref{eq2}) (the generatrix lines parallel to the $(x_1,x_2,x_3)$ axes) and a degenerate hyperbola given by (\ref{eq3}). The equation (\ref{eq2}) shows that  $-1/(2\sqrt{3})\leq x_4\leq 1/\sqrt{3}$ while (\ref{eq6}) allows us to determine a discrete set of values of $x_3$ for those possible values of $x_4$. Once particular values of $x_3$ and $x_4$ are determined, equations (\ref{eq1})--(\ref{eq6}) enable us to ensure that the coordinates $x_1,x_2,x_6,x_8,x_5,x_7$ of our surface are contained in the two-dimensional spheres of the form
$$
x_1^2+x_2^2=F_1(x_3,x_4),\qquad  x_6^2+x_8^2=F_2(x_3,x_4),\qquad x_7^2+x_5^2=F_3(x_3,x_4),
$$
for certain functions $F_1,F_2,F_3$. The final expression for the immersed surfaces can be obtained by using the remaining conditions, e.g. (\ref{eq4}), (\ref{eq5}), resulting from (\ref{sol}).

Instead of following the above procedure, let us now analyse the immersed surface by using global and local invariants for surfaces. The parametrization of the surface in terms of the complex variable $\xi=x+iy$ is conformal, namely the Lie derivative in terms of $\partial_{\bar \xi}$ of the fundamental form
$$
I=4(1+x^2+y^2)^{-2}({\rm d}x\otimes {\rm d}x+{\rm d}y\otimes {\rm d}y)=2(1+|\xi|^2)^{-2}({\rm d}\xi\otimes {\rm d}\bar \xi+{\rm d}\bar \xi\otimes {\rm d}\xi),
$$
which is the metric of the sphere, vanishes.

Some simple calculations permit us to obtain that the Gaussian curvature of our immersed surface is $\kappa=2$ while the Euler-Poincar\'e character is $\chi=2$. The mean curvature tensor of the surface, $\mathcal{H}$, is anti-Hermitian and orthogonal, relative to the metric on $\mathfrak{su}(3)$, to $\partial X_0,$ $\bar \partial X_0$ and $X_0$. The norm of $\mathcal{H}$ is 16. As a consequence, the Veronese surface has positive constant mean curvature. Meanwhile, the Willmore functional $Q=2$ indicates that the winding number of a submanifold of the surface is at least equal to $2$. This happens for the projection by $\pi:(x_1,\ldots,x_8)\in \mathbb{R}^8\simeq \mathfrak{su}(3)\mapsto (x_4,x_6,x_8)\in \mathbb{R}^3$ of the Veronese surface.

Let us analyse a second possible immersed surface induced by a different spectral solution. It is well known that $f_1:=(\partial P_0) f_0$ gives rise to a new solution $P_1$ of the $\mathbb{C}P^2$ model. In other words, the function
$$
f_1=\frac{1}{1+|\xi|^2}(-2\bar \xi,\sqrt{2}(1-|\xi|^2),2\xi).
$$
gives rise to the particular solution
$$
P_1=\frac{f_1\otimes f_1^\dagger}{f_1^\dagger f_1}=\frac{1}{(1+|\xi|^2)^2}\left(\begin{array}{ccc}2|\xi|^2&\sqrt{2}(|\xi|^2-1)\bar \xi&-2\bar \xi^2\\\sqrt{2}(|\xi|^2-1)\xi&(|\xi|^2-1)^2&-\sqrt{2}(|\xi|^2-1)\bar \xi\\-2\xi^2&-\sqrt{2}(|\xi|^2-1)\xi&2|\xi|^2\end{array}\right).
$$
This leads to the immersion formula
$$
X_1=-i(P_1+2P_0)+i{\rm Id}_3=
\frac{i}{(1+|\xi|^2)^2}\left(\begin{array}{ccc}|\xi|^2-1&\sqrt{2}\bar\xi&0\\-\sqrt{2}\xi&0&-\sqrt{2}\bar \xi\\0&-\sqrt{2}\xi&-(|\xi|^2-1)\end{array}\right),
$$
which is an anti-Hermitian solution to 
$$
X_1^3+X_1=0.
$$
In the already used Gell-Mann orthonormal basis $\{S_\alpha\}_{\alpha\in \overline{1,8}}$ of $\mathfrak{su}(3)$, the coordinates of the above $X_1$ satisfy
\begin{equation}\label{eq1bis}
x_3=x_4/\sqrt{3},\quad x_6=0,\quad x_8=0,\quad x_1^3+1=0.
\end{equation}
Since $X_1\neq 0$, the set of equations describing a non-singular manifold reduce to 
\begin{eqnarray}
\sum_{i=1}^7x_i^2-x_6^2=1,\label{eq2bis}\\
x_5^2+x_7^2+\frac 23x_4^2=\frac 12,\label{eq3bis}\\
x_1^4-\left(x_2^2+\frac 23x_4^2-\frac 12\right)^2=0\label{eq4bis}.
\end{eqnarray}
In this case, one obtains that the six previous independent conditions  (\ref{eq1bis})--(\ref{eq4bis}) on $\mathfrak{su}(3)$, which is eight-dimensional, give rise to a surface. 
The condition (\ref{eq4bis}) shows that the immersed surface is a closed fourth degree curve in the coordinates $x_1,x_2$ parametrized by 
$$
-\frac{\sqrt{3}}2<x_4\leq \frac{\sqrt{3}}2,
$$
where the above bounds for $x_4$ are a consequence of (\ref{eq3bis}). 

The metric of the surface is that one of the sphere, as in the previous example, with Gaussian curvature $\kappa=2$. However, the mean curvature is $\langle \mathcal{H},\mathcal{H}\rangle=4$, where $\mathcal{H}$ is the mean curvature tensor of the surface. This implies a smaller value of the Willmore functional $W=2\pi$ relative to the Veronese surface, but a topological charge $Q=0$. 

A last anti-holomorphic solution to the $\mathbb{C}P^2$ model can be found
$$
P_2:=i{\rm Id}_3-P_0-P_1=\frac{1}{1+|\xi|^2}\left(\begin{array}{ccc}|\xi|^4&-\sqrt{2}|\xi|^2\bar \xi&\bar \xi^2\\ -\sqrt{2}|\xi|^2\xi&2|\xi|^2&-\sqrt{2}\bar \xi\\ \xi^2&-\sqrt{2}\xi&1\end{array}\right),
$$
which leads to an immersed surface
$$
\begin{aligned}
X_2:&=-i(P_2+2(P_0+P_1))+i\frac 53 {\rm Id}_3\\
&=\frac{i}{(1+|\xi|^2)^2}\left(\begin{array}{ccc}-\frac 13[1-2|\xi|^4+2|\xi|^2]&-\sqrt{2}|\xi|^2\bar \xi&\bar {\xi}^2\\ -\sqrt{2}|\xi|^2\xi&-\frac 13[1+|\xi|^4-4|\xi|^2]&-\sqrt{2}\bar \xi\\\xi^2&-\sqrt{2}\xi&-\frac 13[-2+|\xi|^4+2|\xi|^2]\end{array}\right).
\end{aligned}
$$
Since this surface, which appears as a consequence of the application of the standard techniques for $\mathbb{C}P^{N-1}$ sigma models, is a linear combination of previous ones, its analysis will be not accomplished.

\section{Concluding remarks and outlook}\setcounter{equation}{0}

This paper has presented a cohomological approach to the study of immersed soliton surfaces for integral systems of PDEs through several types of spectral $\mathfrak{g}$-valued differential forms. This has revealed several properties of such surfaces while providing an extension of known methods to general soliton submanifolds for more general systems of integrable PDEs. Our work has also shed some light on the jet approach to the study of immersed soliton surfaces. The $\mathbb{C}P^{N-1}$ sigma models have been studied through our techniques.

In the future, we aim to expand our analysis in several directions. First, it is interesting to investigate whether soliton surfaces are stable and they can therefore be observable in nature. This future study would attempt to devise an appropriate perturbation theory and suggest the development of good approximate solutions. Second, it is worth noting that analytic expressions for surfaces may reveal quality features that otherwise might be difficult to detect numerically. In this respect, we aim to develop computer techniques for the visualization of soliton surfaces. It is also natural to ask how the integrable characteristics -- Hamiltonian structures, conserved quantities, singularities structures and so on -- are manifest on surfaces. Third, we are interested in  problems whose soliton surfaces are well-known experimentally but the associated physical system is not fully developed. Hence, we propose to use the variational problem of the geometric functional, i.e. the Willmore functional interpreted as an action functional, to compute the Euler-Lagrange equations determining the surface. This approach has been accomplished for biological membrane models obtained via the generalized Weierstrass representation \cite{KL99}, but not, to our knowledge, for immersion formulas for soliton surfaces using our approach. Finally, the cohomological approach may be further developed and the properties of the used cohomologies must be studied in detail, e.g. the description of more general Poincar\'e Lemmas. This will be the topic of a future work.


\section{Acknowledgements}
A.M. Grundland was partially supported by the research grant ANR-11LABX-0056-LMHLabEX LMH (France) and
from the NSERC (Canada). J. de Lucas acknowledges partial support from the project MAESTRO DEC-2012/06/A/ST1/00256 of the National Science Center (Poland). This work was partially accomplished during the stay of A.M. Grundland and J. de Lucas at the \'Ecole Normale Superieure de Cachan (CMLA). The authors would also like to thank CMLA for its hospitality and attention during their stay.


\begin{thebibliography}{10}

\bibitem{Ablo82}
Ablowitz MJ 1982 {\sl Nonlinear Phenomena} (Berlin, Springer--Verlag).

\bibitem{BBT06}
Babelon O, Bernard D, and Talon M 2006 {\sl Introduction to Classical Integrable Systems} (Cambridge Monographs on Mathematical Physics) (Cambridge: Cambridge University Press).
 
\bibitem{Bo94}
Bobenko AI 1994 {\it Surfaces in terms of 2 by 2 matrices. Old and new integrable cases} in: {\sl Harmonic Maps and Integrable Systems} (Aspects of Mathematics, vol E 23) (Wiesbaden: Vieweg+Teubner Verlag). 

\bibitem{BE00}
Bobenko A and Eitner U 2000 {\sl Painlev\'e Equations in the Differential Geometry of Surfaces} (Lect. Notes 1753) (Berlin: Springer). 

\bibitem{BJRW88}
Bolton J, Jensen GR, Rigoli M, and Woodward LM 1988 {\sl On conformal
minimal immersions of $S^2$
into $\mathbb{C}P^n$}, Math. Ann. {\bf 279} 599--620.

\bibitem{Br13}
Bracken P 2013 {\it Deformation of surfaces in the three-dimensional space induced by means of integrable systems},
{\sl Nonl. Anal.: Real World Applications} {\bf 14} 1331--1339.


\bibitem{Ca53}
Cartan E 1953 {\it Sur la structure des groupes infinis de transformation} in {\sl Les syst\`emes diff\'erentiels en Involution} (Paris: Gauthier-Villars).

\bibitem{CJZ89}
Chavolin J, Joanny JF and Zinn-Justin J 1989
{\sl Liquids at Interfaces} (Amsterdam: Elsevier)

\bibitem{Ch83}
Chen FF 1983 {\sl Introduction to Plasma Physics and Controlled Fusion} (Plasma Physics vol 1) (New York: Plenum Press).

\bibitem{Ci97}
Cie\'{s}li\'{n}ski J 1997 {\it A generalized formula for integrable classes of surfaces in {L}ie algebras}, {\sl J. Math. Phys.} \textbf{38} 4255--4272. 

\bibitem{Ci07}
Cie\'sli\'nski JL 2007 {\it Pseudospherical surfaces on time scales: a geometric deformation and the spectral approach}, {\sl J. Phys. A} \textbf{40} 12525--12538. 

\bibitem{DGW96}
David F, Ginsparg P, and Zinn-Justin J 1996 {\sl Fluctuating Geometries in Statistical Mechanics and Field Theory} (Amsterdam: North-Holland Publishing Co.).

\bibitem{Da91}
Davidor AS 1991 {\sl Solitons in molecular systems} (New York: Kluwer).

\bibitem{DV00}
Dillen FJE and Verstraelen LCA 2000
{\sl Handbook of Differential Geometry}
(Amsterdam: North--Holland).

\bibitem{DHZ84}
Dim AM, Horv\'ath Z and Zakrzewski WJ 1984 {\it The Riemann-Hilbert problem and finite action $\mathbb{C}P^{N-1}$ solutions}, {\sl Nuclear Phys. B} {\bf 233} 269--288.

\bibitem{DZ80}
Dim AM and Zakrzewski W 1980 {\it General class of solutions in the $\mathbb{C}P^{N-1}$ model}, {\sl Nuclear Phys. B} {\bf 174} 397--406.

\bibitem{DS92}
Doliwa A and Sym A 1992 {\it Constant mean curvature surfaces in $E^3$ as an example of soliton surfaces}, in {\sl Nonlinear Evolution Equations and Dynamical Systems}  (River Edge: World Sci. Publ.) 111--117.

\bibitem{Ei78}
Eichenherr H 1978
{\it $SU(N)$ invariant nonlinear $\sigma$ models}, {\sl Nucl. Phys. B} {\bf 146} 215--223.
 
\bibitem{FG96}
Fokas AS and Gel'fand IM 1996  
{\it Surfaces on Lie groups, on Lie algebras,
and their integrability}, {\sl Comm. Math. Phys.} \textbf{177} 203--220.

\bibitem{FGFL00}
Fokas AS, Gel'fand IM, Finkel F, and Liu QM 2000 
{\it A formula for constructing  infinitely many surfaces on {L}ie algebras and integrable equations}, {\sl Selecta Math.} \textbf{6} 347--375. 

\bibitem{GMS00}
Giachetta G, Mangiarotti L, and Sardanashvily G 2001
{\it Cohomology of the infinite-order jet space and the inverse problem},
{\sl J. Math. Phys.} {\bf 42} 4272--4282.

\bibitem{GG10}
Goldstein PP and Grundland AM 2010 
{\it Invariant recurrence relations for $CP^{N-1}$ models}, {\sl J. Phys. A} \textbf{43} 265206. 

 \bibitem{GP78}
Golo VL and Perelomov AM 1978
{\it Solution of the duality equations for the two-dimensional $SU(N)$-invariant chiral model},
{\sl Phys. Lett. B} {\bf 79} 112--113.

\bibitem{GPW92}
Gross DJ, Piran T, and Weinberg S 1992 {\sl Two-dimensional Quantum Gravity and Random Surfaces} (Singapore: World Scientific).

\bibitem{Grundland15}
Grundland AM 2016 {\it Soliton surfaces in the generalized symmetry approach}, {\sl Theoret. and Math. Phys.} {\bf 188} 1322--1333.

\bibitem{GLM16}
Grundland AM, Levi D, and Martina L 2016
{\it On immersion formulas for soliton surfaces},
{\sl Acta Polytech.} {\bf 56} 180--192.

\bibitem{GP11}
 Grundland AM and Post S 2011 {\it Soliton surfaces associated with generalized symmetries of integrable equations}, {\sl J. Phys. A} {\bf 44} 165203. 


\bibitem{GP12}
Grundland AM and Post S 2012 {\it Surfaces immersed in Lie algebras associated with elliptic integrals}, {\sl J. Phys. A} \textbf{45} 015204.

\bibitem{GP12conf}
Grundland AM and Post S 2012 {\it Soliton surfaces associated with $\mathbb{C}P^{N-1}$ sigma models}, {\sl J. Phys. Conf. Series} {\bf 380} 012023.

\bibitem{GPR14}
Grundland AM, Post S, and Riglioni D 2014 {\it Soliton surfaces and generalized symmetries of integrable systems}, {\sl J. Phys. A} {\bf 47} 015201.

\bibitem{GSZ06}
Grundland AM, Strasburger A, and Zakrzewski WJ 2006 {\it Surfaces immersed in $\mathfrak{su}(N+1)$ Lie algebras obtained from the $\mathbb{C}P^N$ sigma models}, {\sl  J. Phys. A} \textbf{39} 9187--9213.

\bibitem{GY09}
Grundland AM and Yurdu\c sen I 2009 {\it On analytic descriptions of two-dimensional surfaces associated with the $\mathbb{C}P^{N-1}$ sigma model}, {\sl J. Phys. A} {\bf 42} 172001. 

\bibitem{glsgal}
Gubbiotti D, Levi D, and Scimiterna C 2016 {\it Linearizability and fake Lax pairs for nonlinear nonautonomous quad-graph equations subject to a consistency condition during passages around a cubic nucleus}, {\sl Theoret. and Math. Phys.} {\bf 189} 1459--1471. 


\bibitem{He01}
H\'elein F 2001 {\sl Constant Mean Curvature Surfaces, Harmonic Maps and Integrable Systems} (Lectures in Mathematics ETH Z\"urich) (Basel: Birkh\"auser). 

\bibitem{Ho31}
Hopf H 1931 {\it \"Uber die Abbildungen der Dreidimensionalen Sph\"are auf die Kugelfl\"ache}, {\sl Math. Ann. } {\bf 104} 637--665.

 \bibitem{KN96}
Kobayashi S and Nomizu K 1996
{\sl Foundations of Differential Geometry. Vol. I.} (Wiley Classics Library) (New York: John Wiley \& Sons Inc.).

\bibitem{KL99}
Konopelchenko BG and Landolfi G 1999 {\it Generalized Weierstrass representation for surfaces in
multi-dimensional Riemann spaces}, {\sl J. Geom. Phys.} {\bf 29} 319--333.

\bibitem{Ko96}
Konopelchenko BG 1996 {\it Induced surfaces and their integrable dynamics}, {\sl Stud. Appl. Math.} \textbf{96} 9--51.

\bibitem{KV11}
Krasil'shchik J and Verbovetsky A 2011
{\it Geometry of jet spaces and integrable systems},
{\sl J. Geom. Phys.} {\bf 61} 1633--1674.

\bibitem{KL08}
Kruglikov B and V. Lychagin V 2008 {\sl Geometry of differential equations}, in {\sl Handbook of Global Analysis 1214},  (Amsterdam: Elsevier Sci. B) 725--771.

\bibitem{La03}
Landolfi G 2003 {\it New results on the Canham--Helfrich membrane
model via the generalized Weierstrass representation}, {\sl J. Phys. A} {\bf 36} ) 11937--11954. 


\bibitem{lsz90}
Levi D, Sym A, and Tu GZ 1990 {\it A working algorithm to isolate integrable surfaces in $E^3$}, preprint DF INFN 761, ROME-761, 1--10.  

\bibitem{lll10}
Li YQ, Li B, and Lou SY 2010 {\it Constraints for evolution equations with some special forms of Lax pairs and distinguishing Lax pairs by available constraints}, arXiv:1008.1375v2.

\bibitem{Mana06}
Manakov SV and Santini PM 2006 {\it Inverse scattering problem for vectors fields and the Cauchy problem for the heavenly equations}, {\sl Phys. Lett. A} {\bf 359} 613--619.

\bibitem{MS04}
Manton N and Sutcliffe P 2004 {\sl Topological Solitons} (Cambridge Monographs on Mathematical Physics) (Cambridge: Cambridge University Press). 

\bibitem{Ma02}
Marvan M 2002 {\it On the horizontal gauge cohomology and nonremovability of the spectral parameter}, {\sl Acta Appl. Math.} {\bf 72}  51--65. 

\bibitem{m04}
Marvan M 2004 {\it Reducibility of zero curvature representations with application to recursion operators}, {\sl Acta Appl. Math.} {\bf 83}  39--68.

\bibitem{m10}
Marvan M 2010 {\it On the spectral parameter problem}, {\sl Acta Appl. Math.} {\bf 109}  239--255. 

\bibitem{Ma99}
May JP 1999 {\sl A Concise Course in Algebraic Topology} (Chicago Lectures in Mathematics) (Chicago: University of Chicago Press).

\bibitem{Mi86}
Mikhailov AV 1986 {\sl Integrable Magnetic Models Soliton} (Modern Problems in Condensed Matter vol 17) (Amsterdam: North-Holland) 623--690.

\bibitem{MSS91}
 Mikhailov AV, Shabat AB, and Sokolov VV 1991 {\it The symmetry approach to classification of integrable equations. What is integrability?} in {\sl Nonlinear Dynamics} (Berlin: Springer) 115--184.
 
\bibitem{NPW92}
Nelson D, Piran T, and Weinberg S 1992 {\sl Statistical Mechanics of Membranes and Surfaces} (Singapore: World scientific).
 
\bibitem{Ol93}
Olver PJ 1993 {\sl Applications of Lie Groups to Differential Equations} (New York, Springer). 

\bibitem{OLX99}
Ou-Yang Z, Liu J, and Xie Y 1999 {\sl Geometric methods in elastic theory of membranes in liquid crystal phases} (Singapore: World Scientific).

\bibitem{PS91}
Polchinski J and Strominger A 1991 {\it Effective string theory}, {\sl Phys. Rev. Lett.} {\bf 67} 1681--1684.

\bibitem{RS00}
Rogers C and Schief WK 2000 {\sl B\"acklund and Darboux Transformations. Geometry and Modern Applications in Soliton Theory} (Cambridge Texts in Applied Mathematics) (Cambridge: Cambridge University Press). 

\bibitem{Sa94}
Safran SA 2003 {\sl Statistical Thermodynamics of surfaces, interfaces, and membranes} (Frontiers of Physics {\bf 90}) (Boulder: Westview Press).


\bibitem{s02}
Sakovich SY 2004 {\it Cyclic bases of zero-curvature representations: five illustrations to one concept}, {\sl Acta Appl. Math.} {\bf 83} 69--83.

\bibitem{Sa83}
Sasaki JR 1983 {\it General class of solutions of the complex Grassmannian and $\mathbb{C}P^{N-1}$ model}, {\sl Phys. Lett. B} {\bf 130} 69--72.

\bibitem{So52}
Sommerfeld A 1952 {\sl Lectures on theoretical physics Vol 1-3} (New York: Academic Press).

\bibitem{Sy82}
Sym A 1982 {\it Soliton surfaces}, {\sl Lett. Nuovo Cimento} \textbf{33} 394--400.
  
\bibitem{Sy95}
Sym A 1995 {\it Soliton surfaces and their applications (Soliton geometry from spectral problems)} in {\sl Geometric Aspect of the Einstein Equation and Integrable Systems} (Lectures Notes in Physics {\bf 239}) (Berlin: Springer) 154--231. 

\bibitem{Ta95}
Tafel J 1995 {\it Surfaces in $\mathbb{R}^3$ with prescribed curvature}, {\sl J. Geom. Phys.} {\bf 17}, 381--390. 

\bibitem{Ur03}
Urbantke HK 2003 {\it The Hopf fibration-seven times in physics}, {\sl J. Geom. Phys.} {\bf 46} 125--150. 

\bibitem{Vi01}
Vinogradov AM 2001
{\sl Cohomological Analysis of Partial Differential Equations and Secondary Calculus} (Translations of Mathematical Monographs, 204) (Providence: American Mathematical Society).

\bibitem{KV99}
Vinogradov AM and Krasil'shchik IS 1999
{\sl Symmetries and Conservation Laws for Differential Equations of Mathematical Physics}
(Trans. Math. Monographs 182) (Providence: American Mathematical Society).

\bibitem{Xiu16}
Guo XR 2016 {\it Three new $(2+1)-$dimensional integrable systems and some related Darboux transformations}, {\sl Commun. Theor. Phys.} {\bf 65} 735--742.

\bibitem{Zakh94}
Zakharov VE 1994 {\it Dispersionless limit of integrable systems in 2+1 dimensions}, in {\sl Singular Limits of Dispersive Waves} (NATO Adv. Sci. Inst. Ser. B Phys., {\bf 320}) (New York: Plenum).

\bibitem{ZM79}
Zakharov VE and Mikhailov AV 1978 {\it Relativistically invariant two-dimensional models of field theory which are integrable by means of the inverse scattering problem method}, {\sl Sov. Phys. -- JETP} \textbf{74} 1953--1973.
 
\bibitem{Zakrzewski89}
Zakrzewski WJ 1989 {\sl Low-dimensional Sigma Models} (Bristol: Adam Hilger, Ltd.).

\end{thebibliography}

\end{document}